\documentclass{article}
\usepackage[accepted]{icml2022}
\usepackage{times}

\usepackage[utf8]{inputenc} 
\usepackage[T1]{fontenc}    
\usepackage{hyperref}       
\hypersetup{
    colorlinks = true,
    linkbordercolor = {blue},
    citecolor = {blue}
}
\usepackage{url}            
\usepackage{xurl}
\usepackage{booktabs}       
\usepackage{amsfonts}       
\usepackage{nicefrac}       
\usepackage{microtype}      
\usepackage{enumerate}
\usepackage{amsmath,color}

\newcommand{\zero}{\mathbf{0}}
\usepackage{graphics,mwe,subcaption,wrapfig}
\usepackage{algorithmic,algorithm}
\usepackage[algo2e]{algorithm2e}
\usepackage{float}
\newfloat{algorithm}{t}{lop}
\usepackage{amsthm}
\usepackage{amssymb}
\usepackage{caption}
\usepackage{cleveref} 
\usepackage{enumitem}

\makeatletter
\newcommand*\dotp{\mathpalette\dotp@{.5}}
\newcommand*\dotp@[2]{\mathbin{\vcenter{\hbox{\scalebox{#2}{$\m@th#1\bullet$}}}}}
\makeatother

\newtheorem{definition}{Definition}[section]
\newtheorem*{definition*}{Definition}

\newtheorem*{fact*}{Fact} 

\newtheorem{thm}{Theorem}

\newtheorem{lemma}[definition]{Lemma}
\newtheorem{proposition}[definition]{Proposition}

\newtheorem{assum}{Assumption}

\newtheorem*{theorem*}{Theorem}
\newtheorem*{lemma*}{Lemma}
\newtheorem*{proposition*}{Proposition}
   
\newtheorem{assumption*}{Assumption}
\newtheorem{condition*}{Condition}
\newtheorem{exercise*}{Exercise}
\newtheorem*{example*}{Example}

\usepackage{comment}
\usepackage{thm-restate}
\usepackage{bbm}
\usepackage{mathtools}
\usepackage{multirow}
\usepackage{xcolor}

\usepackage[noblocks]{authblk}

\usepackage[toc,page,header]{appendix}
\usepackage{minitoc}

\allowdisplaybreaks

\icmltitlerunning{Submission and Formatting Instructions for ICML 2023}

\begin{document}
\onecolumn
\title{Generalized-Smooth Nonconvex Optimization is As Efficient As Smooth Nonconvex Optimization}
\icmlsetsymbol{equal}{*}

\author[1]{\textit{Ziyi Chen}}
\author[1]{\textit{Yi Zhou}}
\author[1]{\textit{Yingbin Liang}}
\author[3]{\textit{Zhaosong Lu}}

\affil[1]{Department of Electrical and Computer Engineering, University of Utah, Salt Lake City, UT, US}
\affil[2]{Department of Electrical and Computer Engineering, Ohio State University, Columbus, OH, US}
\affil[3]{Department of Industrial and Systems Engineering, University of Minnesota, Minneapolis, MN, US}

\affil[1]{\small {Email: \{u1276972,yi.zhou\}@utah.edu}}
\affil[2]{\small {Email: liang.889@osu.edu}}
\affil[3]{\small {Email: zhaosong@umn.edu}}
\maketitle


\doparttoc 
\faketableofcontents 

\begin{abstract}
Various optimal gradient-based algorithms have been developed for smooth nonconvex optimization. However, many nonconvex machine learning problems do not belong to the class of smooth functions and therefore the existing algorithms are sub-optimal. Instead, these problems have been shown to satisfy certain generalized-smooth conditions, which have not been well understood in the existing literature. In this paper, we propose a notion of $\alpha$-symmetric generalized-smoothness that extends the existing notions and covers many important functions such as high-order polynomials and exponential functions. We study the fundamental properties and establish descent lemmas for the functions in this class. Then, to solve such a large class of nonconvex problems, we design a special deterministic normalized gradient descent algorithm that achieves the optimal iteration complexity $\mathcal{O}(\epsilon^{-2})$, and also prove that the popular SPIDER variance reduction algorithm achieves the optimal sample complexity $\mathcal{O}(\epsilon^{-3})$ in the stochastic setting. Our results show that solving generalized-smooth nonconvex problems is as efficient as solving smooth nonconvex problems.
\end{abstract}

\section{Introduction}
In many modern machine learning applications, training machine learning model requires solving a nonconvex optimization problem with big data, for which many efficient gradient-based optimization algorithms have been developed, e.g., gradient descent (GD) \citep{carmon2020lower}, stochastic gradient descent (SGD) \citep{ghadimi2013stochastic} and many advanced stochastic variance reduction algorithms \citep{fang2018SPIDER,wang2019spiderboost}. In particular, the complexities of these algorithms have been extensively studied in nonconvex optimization. Specifically, under the standard assumption that the objective function is $L$-smooth (i.e., has Lipschitz continuous gradient), it has been shown that the basic GD algorithm \citep{carmon2020lower} and many advanced stochastic variance reduction algorithms  \citep{fang2018SPIDER,cutkosky2019momentum} achieve the complexity lower bounds of finding an approximate stationary point of deterministic nonconvex optimization and stochastic nonconvex optimization, respectively. \footnote{Deterministic and stochastic optimization problems are formulated respectively as $\min_{w} f(w)$ and $\min_{w} \mathbb{E}_{\xi\sim\mathbb{P}} f_{\xi}(w)$.}


Although the class of smooth nonconvex problems can be effectively solved by the above provably optimal algorithms, it does not include many important modern machine learning applications, e.g., distributionally robust optimization (DRO) \citep{jin2021non} and language model learning \citep{zhang2019gradient}, etc. Specifically, for the problems involved in these applications, they are not globally smooth but have been shown to satisfy certain generalized-smooth conditions, in which the smoothness parameters scale with the gradient norm in various ways (see the formal definitions in Section \ref{sec:smoothlits}). To solve these generalized-smooth-type nonconvex problems, the existing works have developed various gradient-based algorithms, but only with sub-optimal complexity results for stochastic optimization. Therefore, we are motivated to {\em systematically build a comprehensive understanding of generalized-smooth functions and develop algorithms with improved complexities}.

To achieve this overarching goal, we need to address several fundamental challenges. First, the existing generalized-smooth conditions are proposed for specific application examples. Therefore, they define relatively restricted classes of functions that do not cover many popular ones such as high-order polynomials and exponential functions. Thus, we are motivated to consider the following question.

\begin{itemize}[leftmargin=*,topsep=0mm,itemsep=1mm]
	\item {\em Q1: How to extend the existing notion of generalized-smoothness to cover a broad range of functions used in machine learning practice? What are the fundamental properties of the functions in this class?}
\end{itemize}

Second, for such an extended class of generalized-smooth problems, it is expected that first-order algorithms may generally suffer from higher computation complexity (as compared to solving smooth problems). On the other hand, it is unclear how to design first-order algorithms that can efficiently solve these more challenging problems. Therefore, we aim to answer the following question.

\begin{itemize}[leftmargin=*,topsep=0mm,itemsep=1mm]
	\item {\em Q2: Can first-order algorithms solve generalized-smooth nonconvex problems as efficiently as solving smooth nonconvex problems? In particular, what algorithms can achieve the optimal complexities?}
\end{itemize}

\subsection{Our Contribution}
In this paper, we provide comprehensive and affirmative answers to the aforementioned fundamental questions. Our contributions are summarized as follows. 
\begin{itemize}[leftmargin=*,topsep=0mm,itemsep=1mm]
	\item We propose a class of $\alpha$-symmetric generalized-smooth functions, denoted by $\mathbf{\mathcal{L}}_{\text{sym}}^*(\alpha)$, which we show strictly contains the popular class of $L$-smooth functions (i.e., functions with Lipschitz continuous gradient), the class of asymmetric generalized-smooth functions \citep{levy2020large,jin2021non} and the class of Hessian-based generalized-smooth functions \citep{zhang2019gradient} (see the definitions in Section \ref{sec:smoothlits}). In particular, we show that our proposed function class $\mathbf{\mathcal{L}}_{\text{sym}}^*(\alpha)$ includes a wide range of popular machine learning problems and functions used in practice, including distributionally robust optimization \citep{levy2020large,jin2021non}, objective function of language models \citep{zhang2019gradient}, high-order polynomials and exponential functions.
	
	\item We study the fundamental properties of functions in the class $\mathbf{\mathcal{L}}_{\text{sym}}^*(\alpha)$ and establish new decent lemmas for functions in $\mathbf{\mathcal{L}}_{\text{sym}}^*(\alpha)$ with different values of $\alpha$ (See Proposition \ref{prop:PS_equiv}). These technical tools play an important role later in designing new gradient-based algorithms and developing their corresponding convergence analysis. 
	
	\item We develop a $\beta$-normalized gradient descent (named $\beta$-GD) algorithm for solving nonconvex problems in $\mathcal{L}_{\text{sym}}^*(\alpha)$, which normalizes the gradient $\nabla f(w_t)$ with the factor $\|\nabla f(w_t)\|^\beta$ in each iteration. We show that $\beta$-GD finds an approximate stationary point $\mathbb{E}\|\nabla f(w)\|\le \epsilon$ with iteration complexity $\mathcal{O}(\epsilon^{-2})$ as long as $\alpha\le\beta\le1$, which matches the iteration complexity lower bound for deterministic smooth nonconvex optimization and hence is an optimal algorithm. On the other hand, we show that it may diverge when $0<\beta<\alpha$ is used. 
	
	\item For nonconvex stochastic optimization, we propose a class of expected $\alpha$-symmetric generalized-smooth functions, denoted by $\mathbb{E}\mathcal{L}_{\text{sym}}^*(\alpha)$, which substantially generalizes the popular class of expected smooth functions. Interestingly, we prove that the original SPIDER algorithm still achieves the optimal sample complexity $\mathcal{O}(\epsilon^{-3})$ for solving nonconvex stochastic problems in $\mathbb{E}\mathcal{L}_{\text{sym}}^*(\alpha)$. 
\end{itemize}

In summary, our work reveals that generalized-smooth nonconvex (stochastic) optimization is as efficient as smooth nonconvex (stochastic) optimization, and the optimal complexities can be achieved by $\beta$-GD (for deterministic case) and SPIDER (for stochastic case), respectively.

\subsection{Related Work}


\textbf{$L$-smooth Functions $\mathcal{L}$: } For deterministic nonconvex $L$-smooth problems, it is well-known that GD achieves the optimal iteration complexity $\mathcal{O}(\epsilon^{-2})$ \citep{carmon2020lower}. For stochastic nonconvex problems $f(w):=\mathbb{E}_{\xi\sim\mathbb{P}}f_{\xi}(w)$, SGD achieves $\mathcal{O}(\epsilon^{-4})$ sample complexity \citep{ghadimi2013stochastic} which has been proved optimal for first-order stochastic algorithms if only the population loss $f$ is $L$-smooth \citep{arjevani2022lower}. \cite{fang2018SPIDER}  proposed the first variance reduction algorithm named SPIDER that achieves the optimal sample complexity $\mathcal{O}(\epsilon^{-3})$ under the stronger expected smoothness assumption (see eq. \eqref{ESG} for its definition). At the same time, several other variance reduction algorithms have been developed for stochastic nonconvex optimization that achieve the optimal sample complexity. For example, SARAH \citep{nguyen2017sarah} and SpiderBoost  \citep{wang2019spiderboost} can be seen as unnormalized versions of SPIDER. STORM further improved the practical efficiency of these algorithms by using single-loop updates with adaptive learning rates \citep{cutkosky2019momentum}. \cite{zhou2020stochastic} proposed the SNVRG algorithm by adjusting the SVRG variance reduction technique \citep{johnson2013accelerating,reddi2016stochastic} using multiple nested reference points, which also converge to a second-order stationary point. 

\textbf{Hessian-based Generalized-smooth Functions $\mathcal{L}_{\text{H}}^*$: } \cite{zhang2019gradient} extended the $L$-smooth function class to a Hessian-based generalized-smooth function class $\mathcal{L}_{\text{H}}^*$ which allows the Lipschitz constant to linearly increase with the gradient norm (see Definition \ref{def:H}) and thus includes higher-order polynomials and many language models that are not $L$-smooth. For objective function on $\mathcal{L}_{\text{H}}^*$, \cite{zhang2019gradient} also proposed clipped GD and normalized GD which keep the optimal iteration complexity $\mathcal{O}(\epsilon^{-2})$, and proposed clipped SGD which also achieves sample complexity $\mathcal{O}(\epsilon^{-4})$. \cite{zhang2020improved} proposed a general framework for clipped GD/SGD with momentum acceleration and obtained the same complexities for both deterministic and stochastic optimization. \cite{zhao2021convergence} obtained sample complexity $\mathcal{O}(\epsilon^{-4})$ for normalized SGD with both small constant stepsize and diminishing stepsize. A contemporary work \cite{reisizadeh2023variance} reduced the sample complexity to $\mathcal{O}(\epsilon^{-3})$ by combining SPIDER variance reduction technique with gradient clipping.

\textbf{Asymmetric Generalized-Smooth Functions $\mathbf{\mathcal{L}}_{\text{asym}}^*$: } Variants of clipped/normalized GD and SGD have been proposed on the asymmetric generalized-smooth function class $\mathbf{\mathcal{L}}_{\text{asym}}^*$, which looks like a first-order variant of $\mathcal{L}_{\text{H}}^*$ (see Definition \ref{def:asym}). For example, \cite{jin2021non} applied mini-batch normalized SGD with momentum proposed by \citep{cutkosky2020momentum} to distributionally robust optimization problem which has been proved equivalent to minimizing a function in $\mathbf{\mathcal{L}}_{\text{asym}}^*$
\cite{levy2020large,jin2021non}, and also obtained sample complexity $\mathcal{O}(\epsilon^{-4})$. \cite{yang2022normalized} made normalized and clipped SGD differentially private by adding Gaussian noise. \cite{crawshawrobustness2022} proposed generalized signSGD with ADAM-type normalization and obtained sample complexity $\mathcal{O}(\epsilon^{-4})$ on a smaller coordinate-wise version of $\mathbf{\mathcal{L}}_{\text{asym}}^*$. 

\section{Existing Notions of Generalized-Smoothness}\label{sec:smoothlits}

The class of $L$-smooth functions, which we denote as $\mathcal{L}$, includes all continuously differentiable functions with Lipschitz continuous gradient. Specifically, for any $f\in \mathcal{L}$, there exists $L_0>0$ such that 
\begin{align}
	\|\nabla f(w')-\nabla f(w)\| \le L_0 \|w'-w\|, ~~\forall w, w' \in \mathbb{R}^d. \label{eq: smooth}
\end{align}

Many useful functions fall into this class, e.g., quadratic functions, logistic functions, etc. Nevertheless, $\mathcal{L}$ is a restricted function class that cannot efficiently model a broad class of functions, including higher-order polynomials, exponential functions, etc. For example, consider the one-dimensional polynomial function $f(x)=x^4$ in the range $x\in [-10, 10]$. According to \eqref{eq: smooth}, its smoothness parameter $L_0$ can be as large as $1200$, leading to an ill-conditioned problem that hinders optimization. 

To address this issue and provide a better model for optimization, previous works have introduced various notions of generalized-smoothness, which cover a broader class of functions that are used in machine learning applications. For example, distributionally robust optimization (DRO) is an important machine learning problem, and recently it has been proved that DRO can be reformulated as another problem whose objective function belongs to the following asymmetric generalized-smooth function class ($\mathbf{\mathcal{L}}_{\text{asym}}^*$) \citep{levy2020large,jin2021non}. 

\begin{definition}[$\mathbf{\mathcal{L}}_{\text{asym}}^*$ function class]\label{def:asym}
	The asymmetric generalized-smooth function class $\mathcal{L}_{\text{asym}}^*$ is the class of differentiable functions $f:\mathbb{R}^d\to\mathbb{R}$ that satisfy the following condition for all $w, w'\in \mathbb{R}^d$ and some constants $L_0, L_1>0$.
	\begin{align}
		&\|\nabla f(w')-\nabla f(w)\|\le \big( L_0  + L_1\|\nabla f(w')\|\big)\|w'-w\|. \label{asym}
	\end{align} 
\end{definition}
To elaborate, we name the above function class asymmetric generalized-smooth as the definition in \eqref{asym} takes an asymmetric form. In particular,
the smoothness parameter of the functions in $\mathcal{L}_{\text{asym}}^*$ scales with the gradient norm $\|\nabla f(w')\|$. This implies that the nonconvex problem can be ill-conditioned in the initial optimization stage when the gradient is relatively large. 

On the other hand, \cite{zhang2019gradient} showed that high-order polynomials and many language models belong to the following Hessian-based generalized-smooth function class $\mathbf{\mathcal{L}}_{\text{H}}^*$.

\begin{definition}[$\mathbf{\mathcal{L}}_{\text{H}}^*$ function class]\label{def:H}
	The Hessian-based generalized-smooth function class $\mathbf{\mathcal{L}}_{\text{H}}^*$ is the class of twice-differentiable functions $f:\mathbb{R}^d\to\mathbb{R}$ that satisfy the following condition for all $w\in \mathbb{R}^d$ and some constants $L_0, L_1>0$.
	\begin{align}
		&\|\nabla^2 f(w)\|\le L_0+L_1\|\nabla f(w)\|.\label{H}
	\end{align} 
\end{definition}

In addition to the above notions of generalized-smoothness, many other works have developed optimization algorithms for minimizing the class of higher-order smooth functions, i.e., functions with Lipschitz continuous higher-order gradients \citep{Nesterov2006,carmon2020lower,carmon2021lower}. However, the resulting algorithms usually require either computing higher-order gradients or solving higher-order subproblems, which are not suitable for machine learning applications with big data. In the following subsection, we propose a so-called $\alpha$-symmetric generalized-smooth function class, which we show substantially generalizes the existing generalized-smooth function classes and covers a wide range of functions used in many important machine learning applications. 

\section{The $\alpha$-Symmetric Generalized-Smooth Function Class}
We propose the following class of $\alpha$-symmetric generalized-smooth functions $\mathbf{\mathcal{L}}_{\text{sym}}^*(\alpha)$, which we show later covers the aforementioned generalized-smooth function classes and includes many important machine learning problems. Throughout the whole paper, we define $0^0=1$.


\begin{definition}[$\mathbf{\mathcal{L}}_{\text{sym}}^*(\alpha)$ function class]\label{def:P}
	For $\alpha \in [0,1]$, the $\alpha$-symmetric generalized-smooth function class $\mathbf{\mathcal{L}}_{\text{sym}}^*(\alpha)$ is the class of differentiable functions $f:\mathbb{R}^d\to\mathbb{R}$ that satisfy the following condition for all $w, w'\in \mathbb{R}^d$ and some constants $L_0, L_1>0$.
	\begin{align}
		&\|\nabla f(w')-\nabla f(w)\| \le \big(L_0+L_1\max_{\theta\in[0,1]}\|\nabla f(w_{\theta}(w,w'))\|^{\alpha}\big) \|w'-w\|,  \label{P}
	\end{align}
	where $w_{\theta}(w,w'):=\theta w'+(1-\theta)w$. 
\end{definition}
\textbf{Remark:} we use $w_{\theta}(w,w')$ to emphasize its dependence on $w, w'$. Later whenever $w, w'$ is given, we will use the abbreviation $w_{\theta}$.

It can be seen that the above function class $\mathbf{\mathcal{L}}_{\text{sym}}^*(\alpha)$ covers the aforementioned function classes $\mathcal{L}$ (corresponds to $L_1=0$) and $\mathbf{\mathcal{L}}_{\text{asym}}^*$ (with $L_1>0, \alpha=1$ and $\max_{\theta\in[0,1]}\|\nabla f(w_{\theta}(w,w'))\|$ being replaced with the smaller term $\|\nabla f(w)\|$). In particular, compared to the asymmetric generalized-smooth function class $\mathbf{\mathcal{L}}_{\text{asym}}^*$, our proposed function class $\mathbf{\mathcal{L}}_{\text{sym}}^*(\alpha)$ generalizes it in two aspects. First, $\mathbf{\mathcal{L}}_{\text{sym}}^*(\alpha)$ defines generalized-smoothness in a symmetric way with regard to the points $w$ and $w'$ since it considers the maximum gradient norm over the line segment $\{w_{\theta}: \theta\in [0,1]\}$. As a comparison, $\mathbf{\mathcal{L}}_{\text{asym}}^*$ defines generalized-smoothness in an asymmetric way. Second, $\mathbf{\mathcal{L}}_{\text{sym}}^*(\alpha)$ covers the functions whose smoothness parameter can scale polynomially as $\max_{\theta\in[0,1]}\|\nabla f(w_{\theta})\|^{\alpha}$, whereas $\mathbf{\mathcal{L}}_{\text{asym}}^*$ only considers the special case $\alpha=1$. 

Next, we show connections among all these generalized-smooth function classes, and prove that our proposed function class $\mathbf{\mathcal{L}}_{\text{sym}}^*(\alpha)$ is substantially bigger than others.

\begin{thm}[Function class comparison]\label{thm: func_class}
	The generalized-smooth function classes $\mathbf{\mathcal{L}}_{\text{asym}}^*$, $\mathbf{\mathcal{L}}_{\textsc{H}}^*$ and $\mathbf{\mathcal{L}}_{\text{sym}}^*(\alpha)$ satisfy the following properties.
	\begin{enumerate}[leftmargin=*,topsep=0mm,itemsep=1mm]
		\item $\mathcal{L}_{\text{asym}}^* \subset \mathcal{L}_{\text{sym}}^*(1)$; \label{item:asym}
		\item $\mathbf{\mathcal{L}}_{\textsc{H}}^* \subset \mathbf{\mathcal{L}}_{\text{sym}}^*(1)$. Moreover, they are equivalent when restricted to the set of twice-differentiable functions;\label{item:H}
		\item The polynomial function $f(w)=|w|^{\frac{2-\alpha}{1-\alpha}}, w\in\mathbb{R}, \alpha\in (0,1)$ satisfies $f\in\mathcal{L}_{\text{sym}}^*(\alpha)$. However, $f\not\in \mathcal{L}_{\text{sym}}^*(\widetilde{\alpha})$ for all $\widetilde{\alpha} \in (0,\alpha)$ and $f\not\in \mathcal{L}_{\text{asym}}^*$;\label{item:sym_poly}
		\item The exponential function $f(w)=e^{w}+e^{-w}, w\in\mathbb{R}$ satisfies $f\in\mathcal{L}_{\text{sym}}^*(1)$. However, $f\not\in \mathcal{L}_{\text{sym}}^*(\widetilde{\alpha})$ for all $\widetilde{\alpha} \in (0,1)$ and $f\not\in \mathcal{L}_{\text{asym}}^*$.\label{item:sym_exp}
	\end{enumerate}
\end{thm}

\textbf{Remark:} The functions in items 3 \& 4 can be generalized to high-dimensional case $w\in\mathbb{R}^d$ by using $f(w)=\|w\|^{\frac{2-\alpha}{1-\alpha}}$ and $f(w)=e^{\|w\|}+e^{-\|w\|}$ respectively.

To elaborate, items 1 \& 2 show that a special case of our proposed $\alpha$-symmetric generalized-smooth function class $\mathcal{L}_{\text{sym}}^*(1)$ includes the other existing generalized-smooth function classes $\mathcal{L}_{\text{asym}}^*, \mathcal{L}_{\text{H}}^*$. In particular, when $f$ is restricted to be twice-differentiable, the class $\mathcal{L}_{\text{H}}^*$ is equivalent to $\mathcal{L}_{\text{sym}}^*(1)$. Moreover, items 3 \& 4 show that our proposed generalized-smooth function class $\mathcal{L}_{\text{sym}}^*(\alpha)$ includes a wide range of `fast-growing' functions, including high-order polynomials and even exponential functions, which are not included in $\mathcal{L}_{\text{asym}}^*$. In summary, our proposed $\alpha$-symmetric generalized-smooth function class $\mathcal{L}_{\text{sym}}^*(\alpha)$ extends the existing boundary of smooth functions in nonconvex optimization.


Next, for the functions in $\mathcal{L}_{\text{sym}}^*(\alpha)$, we establish various important technical tools that are leveraged later to develop efficient algorithms and their convergence analysis.

\begin{proposition}[Technical tools]\label{prop:PS_equiv}
	The function class $\mathcal{L}_{\text{sym}}^*(\alpha)$ can be equivalently defined as follows.
	\begin{enumerate}[leftmargin=*,topsep=0mm,itemsep=1mm]
		\item\label{item:df_bound_poly} For any $\alpha\in (0,1)$, function $f$ belongs to $\mathcal{L}_{\text{sym}}^*(\alpha)$ if and only if for any $w,w'\in\mathbb{R}^d$,
		\begin{align}
			&\|\nabla f(w')-\nabla f(w)\| \le\|w'-w\| \big(K_0+K_1\|\nabla f(w)\|^{\alpha}
			+K_2\|w'-w\|^{\frac{\alpha}{1-\alpha}}\big).\label{PS_poly_up}
		\end{align}
		where $K_0:=L_0\big(2^{\frac{\alpha^2}{1-\alpha}}+1\big)$, $K_1:=L_1\cdot 2^{\frac{\alpha^2}{1-\alpha}}\cdot 3^{\alpha}$, $K_2:=L_1^{\frac{1}{1-\alpha}}\cdot 2^{\frac{\alpha^2}{1-\alpha}}\cdot 3^{\alpha}(1-\alpha)^{\frac{\alpha}{1-\alpha}}$.
		\item\label{item:df_bound_exp} For $\alpha=1$, function $f$ belongs to $\mathcal{L}_{\text{sym}}^*(1)$ if and only if for any $w,w'\in\mathbb{R}^d$,
		\begin{align}
			&\|\nabla f(w')-\nabla f(w)\| \le \|w'-w\| \big(L_0+L_1\|\nabla f(w)\|\big)\exp\big(L_1\|w'-w\|\big).\label{PS_exp_up}
		\end{align}
	\end{enumerate}
	Consequently, the following descent lemmas hold.
	\begin{enumerate}[leftmargin=*,topsep=0mm,itemsep=1mm]
		\setcounter{enumi}{2}
		\item\label{item:fup_poly} If $f\in \mathcal{L}_{\text{sym}}^*(\alpha)$ for $\alpha\in (0,1)$, then for any $w,w'\in\mathbb{R}^d$,
		\begin{align}
			&f(w') \le f(w)+\nabla f(w)^{\top}(w'-w)+\frac{1}{2}\|w'-w\|^2\big(K_0+K_1\|\nabla f(w)\|^{\alpha}+2K_2\|w'-w\|^{\frac{\alpha}{1-\alpha}}\big).\label{PS_poly_fup}
		\end{align}
		\item\label{item:fup_exp} If $f\in \mathcal{L}_{\text{sym}}^*(1)$, then for any $w,w'\in\mathbb{R}^d$,
		\begin{align}
			&f(w')\le f(w)+\nabla f(w)^{\top}(w'-w)+\frac{1}{2}\|w'-w\|^2\big(L_0+L_1\|\nabla f(w)\|\big)\exp\big(L_1\|w'-w\|\big).\label{PS_exp_fup}
		\end{align}
	\end{enumerate}
\end{proposition}

\textbf{Technical Novelty.} Proving the above items 1 \& 2 turns out to be non-trivial and critical, because they directly imply the items 3 \& 4\footnote{Items 3 \& 4 of Proposition \ref{prop:PS_equiv} can be obtained by substituting items 1 \& 2 into the inequality that $f(w')-f(w)-\nabla f(w)^{\top}(w'-w)\le \int_0^1 |\nabla f(w_{\theta})-\nabla f(w)||w'-w|d\theta$}, which play an important role in the convergence analysis of the algorithms proposed later in this paper. Specifically, there are two major steps to prove the equivalent definitions in items 1 \& 2. First, we prove another equivalent definition, i.e., $f\in\mathcal{L}_{\text{sym}}^*(\alpha)$ if and only if for any $w,w\in\mathbb{R}^d$,
\begin{align}
	&\|\nabla f(w')-\nabla f(w)\|\le \Big(L_0+L_1\int_0^1\|\nabla f(w_{\theta})\|^{\alpha}d\theta\Big)\|w'-w\|. \label{eq: equiv}
\end{align}
Please refer to \eqref{int} in Lemma \ref{lemma:int} in Appendix \ref{supp:lemmas} for the details. To prove this, we uniformly divide the line segment between $w$ and $w'$ into $n$ pieces with the end points $\{w_{\theta}: \theta=\frac{k}{n}\}_{k=0}^n$. Then, we obtain the following bound.
\begin{align*}
	\|\nabla f(w')-\nabla f(w)\|&\le \sum_{k=0}^{n-1} \|\nabla f(w_{(k+1)/n})-\nabla f(w_{k/n})\| \le \|w'-w\| \sum_{k=0}^{n-1} \frac{1}{n}\max_{\theta\in[k/n,(k+1)/n]}h(\theta),
\end{align*}
where $w_{k/n}$ and $w_{(k+1)/n}$ denote $w_{\theta}$ with $\theta=k/n$ and $\theta=(k+1)/n$ respectively, and $h(\theta):=L_0+L_1\|\nabla f(w_{\theta})\|^{\alpha}$. 
As $n\to +\infty$, the summation in the above inequality converges to the desired integral $\int_0^1h(\theta)d\theta$. Second, to prove sufficiency, i.e., \eqref{eq: equiv} implies \eqref{PS_poly_up} \& \eqref{PS_exp_up}, we derive and solve an ordinary differential equation (ODE) of the function $H(\theta):=\int_0^{\theta}h(\theta')d\theta'$. This ODE is obtained by substituting $w'=w_{\theta}$ into the above equivalent definition \eqref{eq: equiv}. Then, to prove necessity, i.e., \eqref{PS_poly_up} \& \eqref{PS_exp_up} imply \eqref{eq: equiv}, we use a similar dividing technique so that averaging the terms $K_0+K_1\|\nabla f(w_{k/n})\|^{\alpha}$ and $L_0+L_1\|\nabla f(w_{k/n})\|$ over $k=0,1,\ldots,n-1$ yields the desired integral as $n\to+\infty$, while at the same time the other terms vanish as $\|w_{(k+1)/n}-w_{k/n}\|^{\frac{\alpha}{1-\alpha}}\to0$ and $\exp\big(L_1\|w_{(k+1)/n}-w_{k/n}\|\big)\to1$. 

Next, we present some nonconvex machine learning examples that belong to the proposed function class $\mathcal{L}_{\text{sym}}^*(\alpha)$.

\textbf{Example 1: Phase Retrieval.}  Phase retrieval is a classic nonconvex machine learning and signal processing problem that arises in X-ray
crystallography and coherent diffraction imaging applications \citep{Drenth1994,Miao1999}. 
In this problem, we aim to recover the structure of a molecular object from far-field diffraction intensity measurements when the object is illuminated by a source light. Mathematically, denote the underlying true object as 
$x\in \mathbb{R}^d$ and suppose we take $m$ intensity measurements, i.e., $y_r = |a_r^{\top}x|^2, r=1,2,...,m$ where $a_r\in\mathbb{R}^d$ and $\top$ denotes transpose. 
Then, phase retrieval proposes to recover the signal by solving the following nonconvex problem.
\begin{align}
	\min_{z\in \mathbb{R}^d} f(z):=\frac{1}{2m}\sum_{r=1}^m (y_r - |a_r^{\top}z|^2)^2.\label{obj_phase}
\end{align}
The above nonconvex objective function is a high-order polynomial in the high-dimensional space. Therefore, it does not belong to the $L$-smooth function class $\mathcal{L}$. In the following result, we formally prove that the above phase retrieval problem can be effectively modeled by our proposed function class $\mathcal{L}_{\text{sym}}^*(\alpha)$.

\begin{proposition}\label{thm_phase}
	The nonconvex phase retrieval objective function $f(z)$ in \eqref{obj_phase} belongs to $\mathcal{L}_{\text{sym}}^*(\frac{2}{3})$.
\end{proposition}


\textbf{Example 2: Distributionally Robust Optimization.}
In many practical machine learning applications, there is usually a gap between training data distribution and test data distribution. Therefore, it is much desired to train a model that is robust to distribution shift. Distributionally robust optimization (DRO) is such a popular optimization framework for training robust models. Specifically, DRO aims to solve the following problem
\begin{align}
	\min_{x \in \mathcal{X}} f(x):=\sup _{Q} \big\{\mathbb{E}_{\xi \sim Q}[\ell_{\xi}(x)] - \lambda d_{\psi}(Q,P)\big\}, \label{DRO}
\end{align}
where the $\psi$-divergence term $\lambda d_{\psi}(Q,P)$ ($\lambda>0$) penalizes the distribution shift between the training data distribution $Q$ and the target distribution $P$, and it takes the form $d_{\psi}(Q,P):=\int\psi\big(\frac{dQ}{dP}\big)dP$. Under mild assumptions on the nonconvex sample loss function $\ell_{\xi}$ (e.g., smooth and bounded variance) and the divergence function $\psi$, the above DRO problem is proven to be equivalent to the following minimization problem \citep{levy2020large,jin2021non}.
\begin{align}
	\min_{x\in \mathcal{X}, \eta \in \mathbb{R}} L(x, \eta):= \lambda \mathbb{E}_{\xi \sim P} \psi^*\left(\frac{\ell_{\xi}(x)-\eta}{\lambda}\right)+\eta, \label{DRO2}
\end{align}
where $\psi^*$ denotes the convex conjugate function of $\psi$. In particular, the objective function $L(x, \eta)$ in the above equivalent form has been shown to belong to the function class $\mathcal{L}_{\text{asym}}^*$ \citep{jin2021non}. Therefore, by item 1 of Theorem \ref{thm: func_class}, we can make the following conclusion.




\begin{lemma}\label{thm:DRO}
	Regarding the equivalent form \eqref{DRO2} of the DRO problem \eqref{DRO}, its objective function $L$ belongs to the function class $\mathcal{L}_{\text{sym}}^*(1)$.
\end{lemma}

\section{Optimal Method for Solving Nonconvex Problems in $\mathcal{L}_{\text{sym}}^*(\alpha)$}
In this section, we develop an efficient and optimal deterministic gradient-based algorithm for minimizing nonconvex functions in $\mathcal{L}_{\text{sym}}^*(\alpha)$ and analyze its iteration complexity. 

The challenge for optimizing the functions in $\mathcal{L}_{\text{sym}}^*(\alpha)$ is that the generalized-smoothness parameter scales with 
$\max_{\theta\in[0,1]}\|\nabla f(w_{\theta})\|^{\alpha}$. To address this issue, we need to use a specialized gradient normalization technique, and this motivates us to consider the  
$\beta$-normalized gradient descent ($\beta$-GD) algorithm as shown in Algorithm \ref{algo: GD}.
To elaborate, $\beta$-GD simply normalizes the gradient update by the gradient norm term $\|\nabla f(w_t)\|^{\beta}$ for some $\beta\ge0$. Such a normalized update is closely related to some existing gradient-type algorithms, including the clipped GD algorithm that uses the normalization term $\max\{\|\nabla f(w_t)\|, C\}$ and the normalized GD that uses the normalization term $\|\nabla f(w_t)\|+C$ \citep{zhang2019gradient}, where $C>0$ is a certain constant. 
\begin{algorithm}
	\caption{$\beta$-Normalized GD}\label{algo: GD}
	{\bf Input:} Iteration number $T$, initialization $w_0$, learning rate $\gamma$, normalization parameter $\beta$.\\
	\For{ $t=0,1, 2,\ldots, T-1$}{ 
		Update $w_{t+1}=w_t-\gamma\frac{\nabla f(w_t)}{\|\nabla f(w_t)\|^{\beta}}$.
	}
	\textbf{Output:} $w_{\widetilde{T}}$ where $\widetilde{T}$ is sampled from $\{0,1,\ldots,T-1\}$ uniformly at random.
\end{algorithm}
We obtain the following convergence result of $\beta$-GD on minimizing functions in $\mathcal{L}_{\text{sym}}^*(\alpha)$. 
\begin{thm}[Convergence of $\beta$-GD]\label{thm:GDconv}
	Apply the $\beta$-GD algorithm to minimize any function $f\in\mathcal{L}_{\text{sym}}^*(\alpha)$ with $\beta\in[\alpha,1]$. Choose $\gamma=\frac{\epsilon^{\beta}}{12(K_0+K_1+2K_2)+1}$\footnote{See the definition of $K_0,K_1,K_2$ in Proposition \ref{prop:PS_equiv}.} if $\alpha\in(0,1)$ and $\gamma=\frac{\epsilon^{\beta}}{4L_0+1}$ if $\alpha=1$  ($\epsilon$ is the target accuracy). Then, the following convergence rate result holds. 
	\begin{align}
		&\mathbb{E}_{\widetilde{T}}\|\nabla f(w_{\widetilde{T}})\|\le\Big(\frac{2}{T\gamma}\Big)^{\frac{1}{2-\beta}}\big(f(w_0)-f^*\big)^{\frac{1}{2-\beta}}+\frac{1}{2}\epsilon.\label{GD_rate}
	\end{align}
	Consequently, to achieve $\mathbb{E}_{\widetilde{T}}\|\nabla f(w_{\widetilde{T}})\|\le \epsilon$, the required overall iteration complexity is $T=\frac{4}{\gamma\epsilon^{2-\beta}}=\mathcal{O}(\epsilon^{-2})$.
\end{thm}

Theorem \ref{thm:GDconv} shows that $\beta$-GD achieves the iteration complexity $\mathcal{O}(\epsilon^{-2})$ when minimizing functions in $\mathcal{L}_{\text{sym}}^*(\alpha)$. Such a complexity result matches the iteration complexity lower bound for deterministic smooth nonconvex optimization and hence is optimal. 
In particular, Theorem \ref{thm:GDconv} shows that to minimize any function $f\in\mathcal{L}_{\text{sym}}^*(\alpha)$, it suffices to apply $\beta$-GD with any $\beta\in[\alpha,1]$ and a proper learning rate $\gamma=\mathcal{O}(\epsilon^\beta)$. Intuitively, with a larger $\alpha$, the gradient norm of function $f$ in the class $\mathcal{L}_{\text{sym}}^*(\alpha)$ increases faster as $\|w\|\to +\infty$, and therefore we need to use a larger normalization parameter $\beta$ and a smaller learning rate $\mathcal{O}(\epsilon^\beta)$ to alleviate gradient explosion. Interestingly, the convergence and iteration complexity of $\beta$-GD remain the same as long as $\beta\ge \alpha$ is used, i.e., over-normalization does not affect the complexity order. In practice, when $\alpha$ is unknown a priori for the function class $\mathcal{L}_{\text{sym}}^*(\alpha)$, one can simply use the conservative choice $\beta = 1$ and is guaranteed to converge.

\textbf{Technical Novelty.} In the proof of Theorem \ref{thm:GDconv}, a major challenge is that due to the $\beta$-normalization term in Algorithm \ref{algo: GD}, the generalized-smoothness of functions in the class $\mathcal{L}_{\text{sym}}^*(\alpha)$ introduces additional higher-order terms to the Taylor expansion upper bounds, as can be seen from the descent lemmas shown in \eqref{PS_poly_fup} (for $\alpha\in(0,1)$) and \eqref{PS_exp_fup} (for $\alpha=1$). In the convergence proof, these terms contribute to certain fast-increasing 
terms that reduce the overall optimization progress. For example, when $\alpha\in(0,1)$, substituting $w=w_t$ and $w'=w_{t+1}$ into \eqref{PS_poly_fup} yields that \footnote{See (i) of \eqref{fdec_poly} in Appendix \ref{sec:GDProof} for the full expression of $\mathcal{O}$ in eq. \eqref{dfo}.}
\begin{align}
	f(w_{t+1})- f(w_t)\le&\gamma\|\nabla f(w_t)\|^{2-\beta}+\frac{\gamma}{6}\Big(\mathcal{O}(\gamma)\|\nabla f(w_t)\|^{2-2\beta}\nonumber\\
	&+\mathcal{O}(\gamma)\|\nabla f(w_t)\|^{2+\alpha-2\beta}+\mathcal{O}(\gamma^{\frac{1}{1-\alpha}})\|\nabla f(w_t)\|^{\frac{(2-\alpha)(1-\beta)}{1-\alpha}}\Big).\label{dfo}
\end{align}
The above key inequality bounds the optimization progress $f(w_{t+1})-f(w_t)$ using gradient norm terms with very different exponents. This makes it challenging to achieve the desired level of optimization progress, as compared with the analysis of minimizing other (generalized) smooth functions in $\mathcal{L}$, $\mathcal{L}_{\text{asym}}^*$ and $\mathcal{L}_{\text{H}}^*$ \cite{zhang2019gradient,jin2021non}. To address this issue and homogenize the diverse exponents, we develop a technical tool in Lemma \ref{lemma:young} in Appendix \ref{supp:lemmas} to bridge polynomials with different exponents. With this technique, we further obtain the following optimization progress bound 
\begin{align}
	f(w_{t+1})-f(w_t)\le-\frac{\gamma}{2}\|\nabla f(w_t)\|^{2-\beta}+\mathcal{O}(\gamma^{\frac{2}{\beta}}),    
\end{align}
which leads to the desired result with proper telescoping.

We also obtain the following complementary result to Theorem \ref{thm:GDconv}, which shows that $\beta$-GD may diverge in general with under-normalization.

\begin{thm}(Divergence of $\beta$-GD)\label{thm: GD-diverge}
	For the $\beta$-GD algorithm with $\beta\in[0,\alpha)$, there always exists a convex function $f\in\mathcal{L}_{\text{sym}}^*(\alpha)$ with a unique minimizer such that for any learning rate $\gamma>0$, $\beta$-GD diverges for all initialization $\|w_0\|>C$ for some constant $C>0$. 
\end{thm}


\section{Expected $\alpha$-Symmetric Generalized-Smooth Functions in Stochastic Optimization}\label{sec:EL}

In this section, we propose a class of  expected $\alpha$-symmetric generalized-smooth functions 
and study their properties in stochastic optimization. Specifically, we consider the following nonconvex stochastic optimization problem 
\begin{align}
	\min_{w\in\mathbb{R}^d} f(w):=\mathbb{E}_{\xi\sim\mathbb{P}} [f_{\xi}(w)],
\end{align}
where $\mathbb{P}$ denotes the distribution of the data sample $\xi$. Throughout, we adopt the following standard assumption on the stochastic gradients \cite{ghadimi2013stochastic,fang2018SPIDER,jin2021non,arjevani2022lower}. 
\begin{assum}\label{assum:var}
	The stochastic gradient is unbiased, i.e., $\mathbb{E}_{\xi\sim\mathbb{P}}[\nabla f_{\xi}(w)]=\nabla f(w)$ and satisfies the following variance bound for some $\Gamma, \Lambda>0$.
	\begin{align}
		\mathbb{E}_{\xi\sim\mathbb{P}}\|\nabla f_{\xi}(w)-\nabla f(w)\|^2\le \Gamma^2\|\nabla f(w)\|^2+\Lambda^2.\label{ESG}
	\end{align}
\end{assum}

If only the population loss $f$ is smooth, i.e., $f\in\mathcal{L}$, a recent work has established a sample complexity lower bound $\mathcal{O}(\epsilon^{-4})$ for first-order stochastic algorithms \citep{arjevani2022lower}, which can be achieved by the standard stochastic gradient descent (SGD) algorithm \citep{ghadimi2013stochastic} and its clipped and normalized versions \citep{zhang2019gradient}. Therefore, one should not expect an improved sample complexity 
when optimizing the larger class of generalized-smooth functions $\mathcal{L}_{\text{sym}}^*(\alpha)$. 
To overcome this sample complexity barrier, many existing works consider the subclass of expected smooth functions $\mathbb{E}\mathcal{L}$, in which there exists a constant $L_0>0$ such that for all $w,w'\in\mathbb{R}^d$, all the functions $f_{\xi}$ satisfy
\begin{align}
	\mathbb{E}_{\xi}\|\nabla f_{\xi}(w')-\nabla f_{\xi}(w)\|^2\le L_0^2\|w'-w\|^2. \label{avg_Lsmooth}
\end{align}
Many variance-reduced algorithms, e.g., SPIDER \cite{fang2018SPIDER} and STORM \cite{cutkosky2019momentum}, have been proved to achieve the near-optimal sample complexity $\mathcal{O}(\epsilon^{-3})$ for optimizing functions in $\mathbb{E}\mathcal{L}$. Therefore, we are inspired to propose and study the following expected $\alpha$-symmetric generalized-smooth function class 
$\mathbb{E}\mathcal{L}_{\text{sym}}^*(\alpha)$.
\begin{definition}[$\mathbb{E}\mathcal{L}_{\text{sym}}^*(\alpha)$ function class]\label{def:avgP}
	For $\alpha \in [0,1]$, the expected $\alpha$-symmetric generalized-smooth function class $\mathbb{E}\mathcal{L}_{\text{sym}}^*(\alpha)$ is the class of differentiable stochastic functions $f=\mathbb{E}_{\xi}[f_{\xi}]$ that satisfy the following condition for all $w, w'\in \mathbb{R}^d$ and some constants $L_0, L_1>0$.
	\begin{align} 
		&\mathbb{E}_{\xi\sim\mathbb{P}}\|\nabla f_{\xi}(w')-\nabla f_{\xi}(w)\|^2\le\|w'-w\|^2\mathbb{E}_{\xi\sim\mathbb{P}}\big(L_0+L_1\max_{\theta\in[0,1]}\|\nabla f_{\xi}(w_{\theta})\|^{\alpha}\big)^2\label{avgP}
	\end{align}
	where $w_{\theta}:=\theta w'+(1-\theta)w$. 
\end{definition}
\textbf{Remark:} It is clear that the function class $\mathbb{E}\mathcal{L}_{\text{sym}}^*(0)$ is equivalent to the function class $\mathbb{E}\mathcal{L}$. 
Also, a sufficient condition to guarantee $f\in\mathbb{E}\mathcal{L}_{\text{sym}}^*(\alpha)$ is that $f_{\xi}\in \mathcal{L}_{\text{sym}}^*(\alpha)$ for every sample $\xi$. 

\begin{proposition}\label{prop:eg_avg}
	Both the aforementioned phase retrieval problem and DRO problem belong to $\mathbb{E}\mathcal{L}_{\text{sym}}^*(\alpha)$ {with $\alpha=\frac{2}{3}, 1$ respectively}.
\end{proposition}


We further develop the following technical tools associated with the function class $\mathbb{E}\mathcal{L}_{\text{sym}}^*(\alpha)$, which are used later to analyze a stochastic algorithm. 

\begin{proposition}[Technical tools]\label{prop:avgsmooth}
	Under Assumption \ref{assum:var}, the following statements hold. 
	\begin{enumerate}[leftmargin=*,topsep=0mm,itemsep=1mm]
		\item\label{item:avg_poly} For any $\alpha\in(0,1)$, function $f=\mathbb{E}_{\xi}[f_{\xi}]$ belongs to $\mathbb{E}\mathcal{L}_{\text{sym}}^*(\alpha)$ if and only if for any $w,w'\in\mathbb{R}^d$,
		\begin{align}
			&\mathbb{E}_{\xi}\|\nabla f_{\xi}(w')-\nabla f_{\xi}(w)\|^2 \le\|w'-w\|^2\big(\overline{K}_0+\overline{K}_1\mathbb{E}_{\xi}\|\nabla f_{\xi}(w)\|^{\alpha}+\overline{K}_2\|w'-w\|^{\frac{\alpha}{1-\alpha}}\big)^2,\label{PSavg_poly_up}
		\end{align}
		where 
		$\overline{K}_0=2^{\frac{2-\alpha}{1-\alpha}}L_0$, $\overline{K}_1=2^{\frac{2-\alpha}{1-\alpha}}L_1$, $\overline{K}_2=(5L_1)^{\frac{1}{1-\alpha}}$;
		\item\label{item:avg_exp} For $\alpha=1$, function $f=\mathbb{E}_{\xi}[f_{\xi}]$ belongs to $\mathbb{E}\mathcal{L}_{\text{sym}}^*(\alpha)$ if and only if for any $w,w'\in\mathbb{R}^d$,
		\begin{align}
			&\mathbb{E}_{\xi}\|\nabla f_{\xi}(w')-\nabla f_{\xi}(w)\|^2\le2\|w'-w\|^2(L_0^2+2L_1^2\mathbb{E}_{\xi}\|\nabla f_{\xi}(w)\|^2)\exp(12L_1^2\|w'-w\|^2).\label{PSavg_exp_up}
		\end{align}         
		\item\label{item:avg2f} $\mathbb{E}\mathcal{L}_{\text{sym}}^*(\alpha)\subset\mathcal{L}_{\text{sym}}^*(\alpha)$.
	\end{enumerate}
\end{proposition}
\textbf{Remark:} Item 3 of Proposition \ref{prop:avgsmooth} implies that we can apply the descent lemmas (items 3 \& 4 of Proposition \ref{prop:PS_equiv}) to the population loss $f=\mathbb{E}_{\xi}[f_{\xi}]$. This is very useful later in the convergence analysis of our proposed stochastic algorithm. 

\section{Optimal Method for Solving Nonconvex Problems in $\mathbb{E}\mathcal{L}_{\text{sym}}^*(\alpha)$}

In this section, we explore stochastic algorithms for solving nonconvex problems in the function class $\mathbb{E}\mathcal{L}_{\text{sym}}^*(\alpha)$ and see if any algorithm can achieve the optimal sample complexity.

In the existing literature, many stochastic variance reduction algorithms, e.g., SPIDER \citep{fang2018SPIDER} and STORM \citep{cutkosky2019momentum}, have been developed and proved to achieve the optimal sample complexity $\mathcal{O}(\epsilon^{-3})$ for minimizing the class of expected-smooth stochastic nonconvex problems (i.e., $\mathbb{E}\mathcal{L}$). However, for the extended class $\mathbb{E}\mathcal{L}_{\text{sym}}^*(\alpha)$, it is unclear what is the sample complexity lower bound and the optimal stochastic algorithm design.  
Inspired by the existing literature and the structures of functions in $\mathbb{E}\mathcal{L}_{\text{sym}}^*(\alpha)$, a good algorithm design must apply both variance reduction and a proper normalization to the stochastic updates in order to combat the generalized-smoothness and achieve an improved sample complexity. Interestingly, we discover that the original SPIDER algorithm design already well balances these two techniques and can be directly applied to solve problems in $\mathbb{E}\mathcal{L}_{\text{sym}}^*(\alpha)$. The original SPIDER algorithm is summarized in Algorithm \ref{algo: SPIDER} below. 


\begin{algorithm}
	\caption{SPIDER \citep{fang2018SPIDER}}\label{algo: SPIDER}
	{\bf Input:} Iteration number $T$, epoch size $q$, initialization $w_0$, learning rate $\gamma$, batchsize $|S_t|$.\\
	\For{ $t=0,1, 2,\ldots, T-1$}{
		Sample a minibatch of data $S_t$. \\
		\eIf{$t\!\mod q=0$}{
			Compute $v_t=\nabla f_{S_t}(w_t)$
		}{
			Compute $v_t=v_{t-1}+\nabla f_{S_t}(w_t)-\nabla f_{S_t}(w_{t-1})$
		}
		Update $w_{t+1}=w_t-\gamma\frac{v_t}{\|v_t\|}$.
		
	}
	\textbf{Output:} $w_{\widetilde{T}}$ where $\widetilde{T}$ is sampled from $\{0,1,\ldots,T-1\}$ uniformly at random..
\end{algorithm}


However, establishing the convergence of SPIDER for the extended function class $\mathbb{E}\mathcal{L}_{\text{sym}}^*(\alpha)$ is fundamentally more challenging. Intuitively, this is because the characterization of variance of the stochastic update $v_t$ is largely affected by the generalized-smoothness structure, and it takes a complex form that needs to be treated carefully. Please refer to the elaboration on technical novelty later for more details. 

Surprisingly, by choosing proper hyper-parameters that are adapted to the function class $\mathbb{E}\mathcal{L}_{\text{sym}}^*(\alpha)$, we are able to prove that SPIDER achieves the optimal sample complexity as formally stated in the following theorem.



\begin{thm}[Convergence of SPIDER]\label{thm:SPIDER}
	Apply the SPIDER algorithm to minimize any function $f=\mathbb{E}_{\xi}[f_{\xi}]\in\mathbb{E}\mathcal{L}_{\text{sym}}^*(\alpha)$ and assume Assumption \ref{assum:var} hold. Set $|S_t|=B$ when $t \!\mod q=0$ and $|S_t|=B'$ otherwise, and let $B\ge \Omega(\max\{\Lambda^2\epsilon^{-2}, \Gamma^2q^2\}),  B' \ge \Omega(\max\{q, q^2\epsilon^2\})$. Choose  $\gamma=\frac{\epsilon}{2\overline{K}_0+4\overline{K}_2+2\overline{K}_1(\Lambda^{\alpha}+\Gamma^{\alpha}+1)+1}$ when $\alpha\in(0,1)$ and $\gamma=\frac{\epsilon}{5L_1\sqrt{\Gamma^2+1}+8\sqrt{L_0^2+2L_1^2\Lambda^2}}$ when $\alpha=1$ ($\epsilon$ is the target accuracy). Then, the following result holds for $T=qK$ iterations where $K\in\mathbb{N}^+$. 
	\begin{align}
		&\mathbb{E}\|\nabla f(w_{\widetilde{T}})\| \le \frac{16}{5T\gamma}\big(\mathbb{E}f(w_0)-f^*\big)+\frac{4\epsilon}{5}.\label{df_rate}
	\end{align}
	In particular, to achieve $\mathbb{E}\|\nabla f(w_{\widetilde{T}})\|\le \epsilon$, we can choose $B=\mathcal{O}(\epsilon^{-2})$, $B'=q=\mathcal{O}(\epsilon^{-1})$, $\gamma=\mathcal{O}(\epsilon)$ and $T=\mathcal{O}(\epsilon^{-2})$\footnote{See eqs. \eqref{q}-\eqref{K} in Appendix \ref{sec:SPIDERProof} for the full expression of these hyperparameters.} so that the above conditions are satisfied. Consequently, the overall sample complexity is $\mathcal{O}(\epsilon^{-3})$. 
\end{thm}

\Cref{thm:SPIDER} proves that SPIDER achieves an overall sample complexity $\mathcal{O}(\epsilon^{-3})$ when solving nonconvex problems in $\mathbb{E}\mathcal{L}_{\text{sym}}^*(\alpha)$ for any $\alpha\in (0,1]$. Note that such a sample complexity matches the well-known sample complexity lower bound for the class of expected-smooth nonconvex optimization problems \cite{fang2018SPIDER}, which is a subset of $\mathbb{E}\mathcal{L}_{\text{sym}}^*(\alpha)$. Consequently, we can make two important observations: (i) this implies that the sample complexity lower bound of $\mathbb{E}\mathcal{L}_{\text{sym}}^*(\alpha)$ is actually $\mathcal{O}(\epsilon^{-3})$; and (ii) the SPIDER algorithm is provably optimal for solving nonconvex problems in such an extended function class. 

\textbf{Technical Novelty.}  Compared with the original analysis of SPIDER for minimizing expected-smooth functions \citep{fang2018SPIDER}, our proof of the above theorem needs to address a major challenge on bounding the expected bias error $\mathbb{E}\|\delta_t\|$ where $\delta_t=v_t-\nabla f(x_t)$. To elaborate more specifically, in the original analysis of SPIDER, \citep{fang2018SPIDER} established the following key lemma (see their Lemma 1) that bounds the martingale variance of the update $v_t$. 
\begin{align}
	\mathbb{E}\|\delta_t\|^2\le \mathbb{E}\|\delta_0\|^2 +\frac{1}{B'}\sum_{k=0}^{t-1}\mathbb{E}_{\xi}\|\nabla f_{\xi}({w_{t+1}}) - \nabla f_{\xi}(w_t)\|^2. \nonumber
\end{align}
The above inequality only depends on the variance reduction structure of SPIDER and hence still holds in our case. However, to further bound the term $\mathbb{E}_{\xi}\|\nabla f_{\xi}({w_{t+1}}) - \nabla f_{\xi}(w_t)\|^2$ for functions in the class $\mathbb{E}\mathcal{L}_{\text{sym}}^*(\alpha)$, we need to leverage the expected generalized-smoothness properties in \eqref{PSavg_poly_up} \& \eqref{PSavg_exp_up} and the update rule $\|w_{t+1}-w_t\|=\gamma=\mathcal{O}(\epsilon)$. We then obtain that $\mathbb{E}_{\xi}\|\nabla f_{\xi}({w_{t+1}}) - \nabla f_{\xi}(w_t)\|^2\le \mathcal{O}(\epsilon^2)+\mathcal{O}(\epsilon^2\|\nabla f(w_t)\|^{2\alpha})$, and consequently, the above martingale variance bound 
becomes
\begin{align}
	\mathbb{E}&\|\delta_t\|^2\le\mathbb{E}\|\delta_0\|^2+\mathcal{O}(\epsilon^2)+\frac{\mathcal{O}(\epsilon^2)}{B'}\sum_{k=0}^{t-1}\mathbb{E}\|\nabla f(w_t)\|^{2\alpha}. \label{martingale2}
\end{align}
When $\alpha>\frac{1}{2}$, the term $\mathbb{E}\|\nabla f(w_t)\|^{2\alpha}$ in \eqref{martingale2} cannot be upper bounded by any functional of $\mathbb{E}\|\nabla f(w_t)\|$, so taking square root of \eqref{martingale2} cannot yield the desired bound $\mathbb{E}\|\delta_t\|\le\mathcal{O}(\epsilon)+\frac{\mathcal{O}(\epsilon)}{\sqrt{B'}}\sum_{k=0}^{t-1}\mathbb{E}\|\nabla f(w_t)\|$ used in the original analysis of SPIDER. 
To address this issue for functions in  $\mathbb{E}\mathcal{L}_{\text{sym}}^*(\alpha)$, we consider the more refined conditional error recursion
\begin{align}
	\mathbb{E}\big(\|\delta_{t+1}\|^2\big|S_{1:t}\big)\le \|\delta_t\|^2+\frac{\epsilon^2}{B'}\big(1+\|\nabla f(w_t)\|^2\big), \label{eq: conditional_variance} 
\end{align}
where there is no randomness in $\|\nabla f(w_t)\|^2$ since we are conditioning on the minibatches $S_{1:t}:=\{S_1,\ldots,S_t\}$ (see \eqref{delta_E2} in Appendix \ref{supp:lemmas}). Therefore, by taking square root of \eqref{eq: conditional_variance} followed by further taking iterated expectation, we can obtain the desired term $\mathbb{E}\|\nabla f(w_t)\|$ in the upper bound of $\mathbb{E}\|\delta_t\|$. After that, we iterate the resulting bound over $t$ via a non-trivial induction argument to complete the analysis (see the proof of \eqref{err_induct} in Lemma \ref{delta_E2} in Appendix \ref{supp:lemmas}). 


\section[title]{Experiments\footnote{The code can be downloaded from \url{https://github.com/changy12/Generalized-Smooth-Nonconvex-Optimization-is-As-Efficient-As-Smooth-Nonconvex-Optimization}}}
\subsection{Application to Nonconvex Phase Retrieval}

In this section, we test our algorithms via solving the nonconvex phase retrieval problem in \eqref{obj_phase}. We set the problem dimension $d=100$ and sample size $m=3000$. The measurement vector $a_r\in\mathbb{R}^d$ and the underlying true parameter $z\in\mathbb{R}^d$ are generated entrywise using Gaussian distribution $\mathcal{N}(0,0.5)$. The initialization $z_0\in\mathbb{R}^d$ is generated entrywise using Gaussian distribution $\mathcal{N}(5,0.5)$. Then, we generate $y_i=|a_r^{\top}z|^2+n_i$ for $i=1,...,m$, where $n_i\sim\mathcal{N}(0,4^2)$ is the additive Gaussian noise. 

\begin{figure}[t]
	\vspace{-5mm}
	\centering
	\begin{subfigure}{0.48\linewidth}   
		\centering 
		\includegraphics[width=\linewidth]{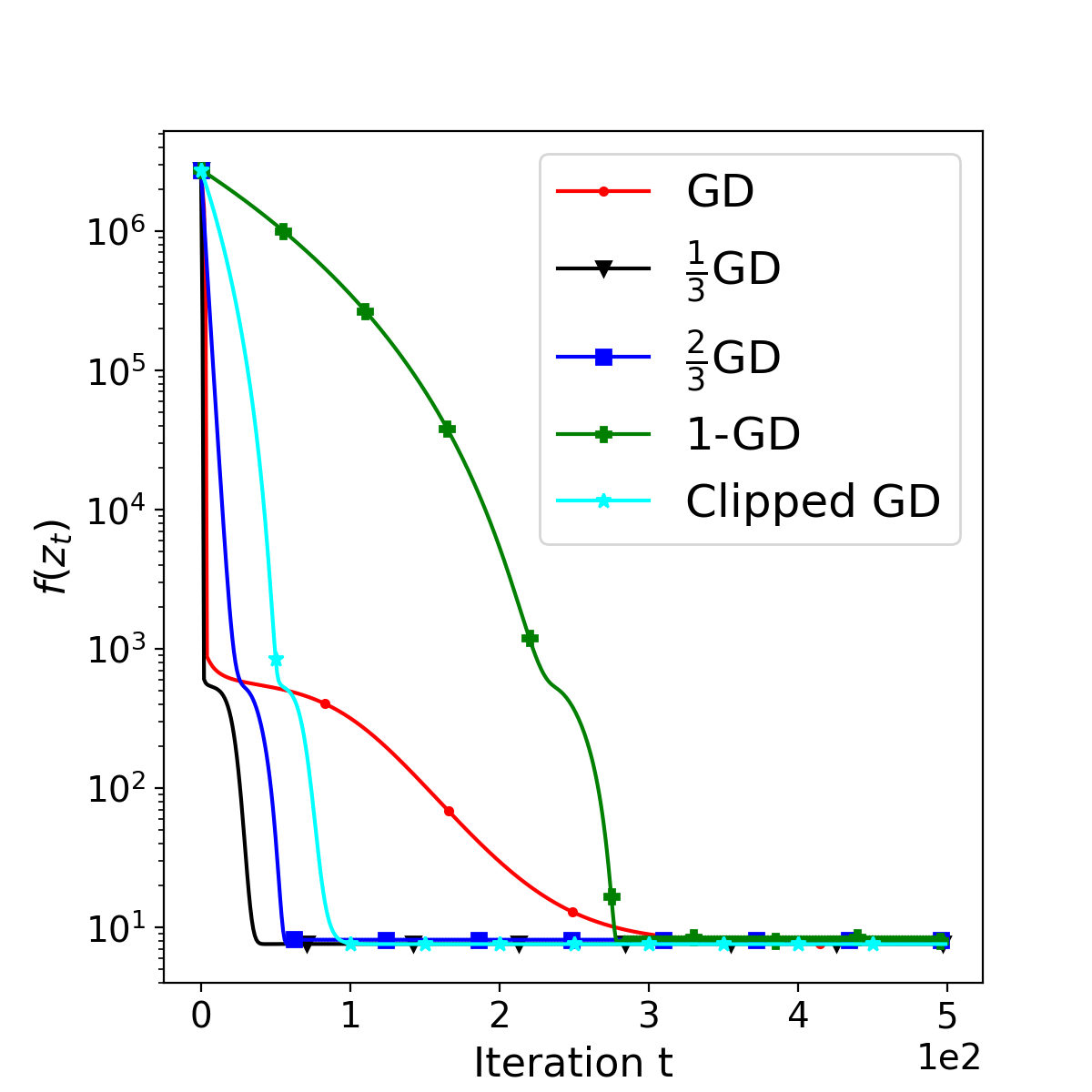}
	\end{subfigure}
	\hspace{-0.02\textwidth}
	\begin{subfigure}{0.48\linewidth}   
		\centering 
		\includegraphics[width=\linewidth]{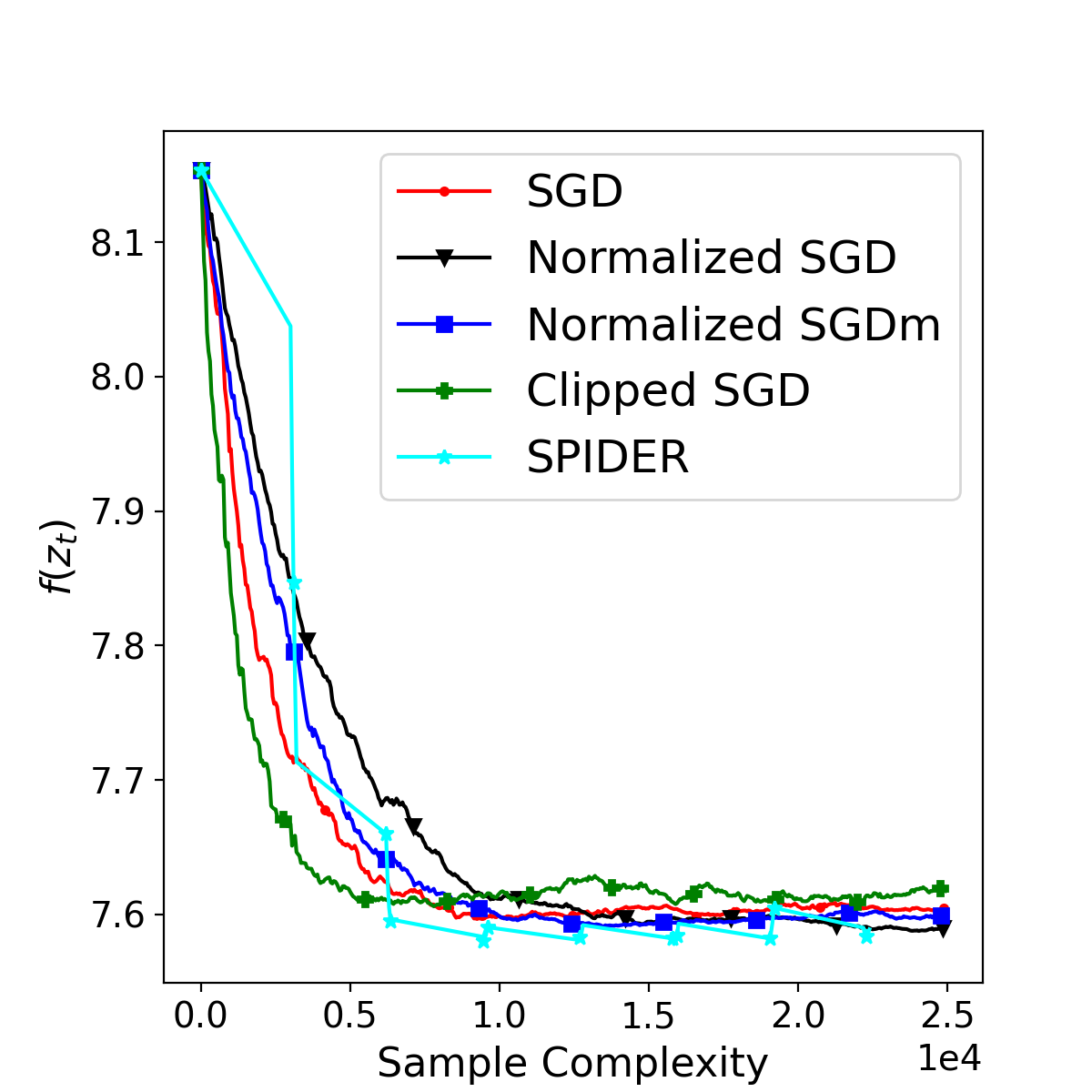}
	\end{subfigure}
	\\
	\begin{subfigure}{0.48\linewidth}   
		\centering 
		\includegraphics[width=\linewidth]{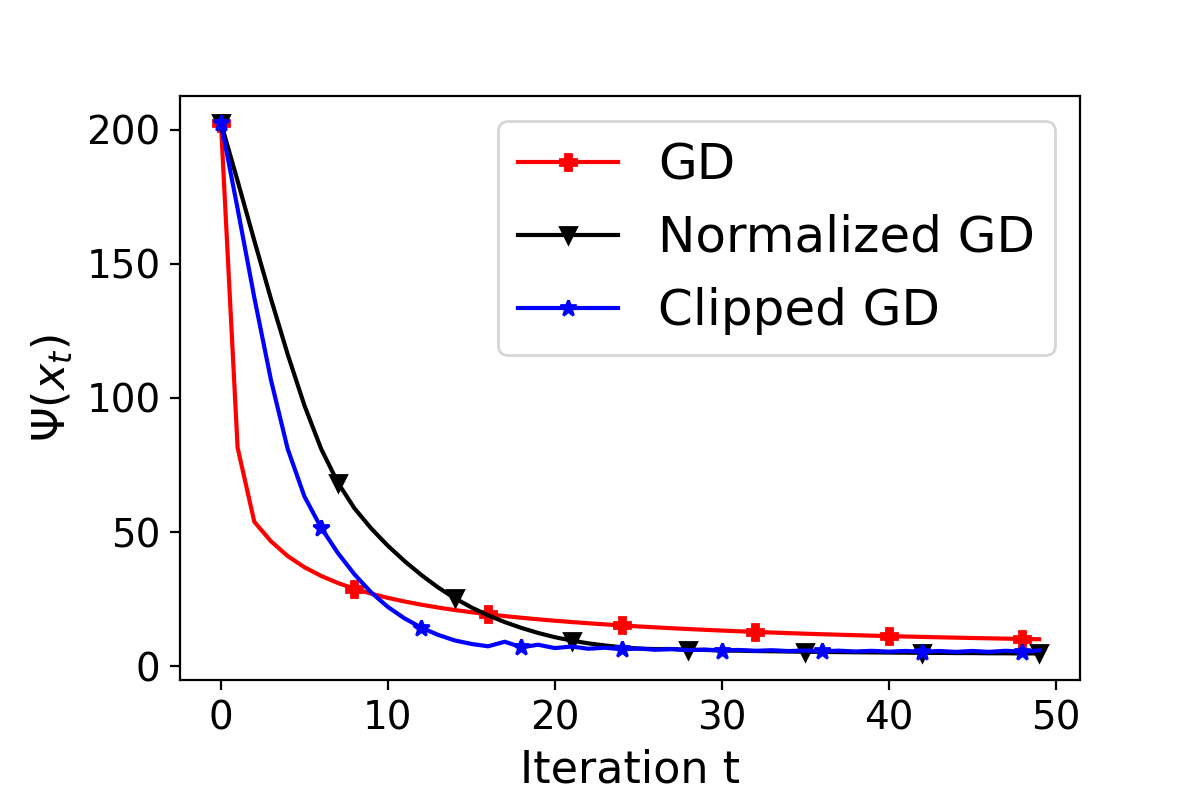}
	\end{subfigure}
	\hspace{-0.02\textwidth}
	\begin{subfigure}{0.48\linewidth}   
		\centering 
		\includegraphics[width=\linewidth]{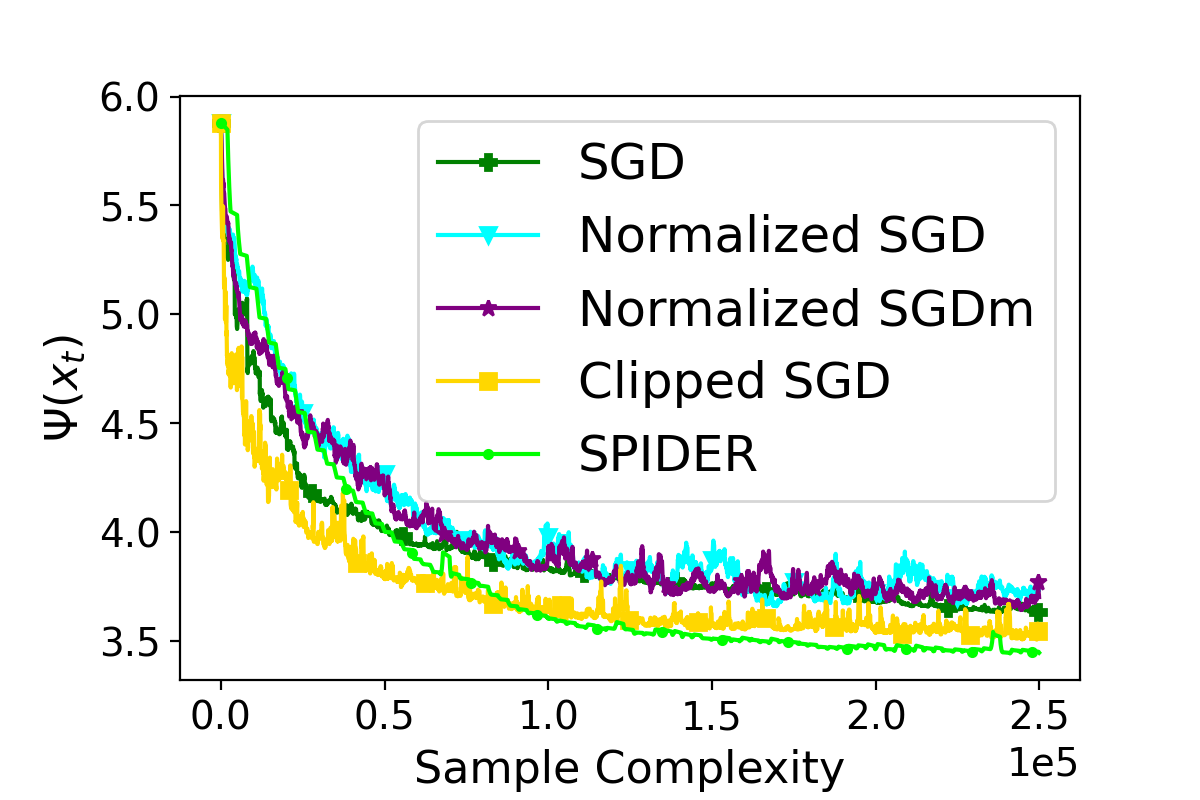}
	\end{subfigure}
	\vspace{0mm}
	\caption{Experimental results (two subfigures above for phase retrieval and two below for DRO).}
	\label{figs}
	\vspace{-5mm}
\end{figure}
We first compare deterministic algorithms with fine-tuned learning rate $\gamma$ over 500 iterations. This includes the basic GD with $\gamma=8\times 10^{-4}$, clipped GD \cite{zhang2019gradient} with $\gamma=0.9$ and normalization term $\max(\|\nabla f(x_t)\|,100)$, and our $\beta$-GD with $\beta = \frac{1}{3}, \frac{2}{3}, 1$ and $\gamma=0.03, 0.1, 0.2$, respectively. Figure \ref{figs} (top left) plots the comparison result on objective function value v.s. iteration. It can be seen that 
our proposed $\beta$-GD with $\beta=\frac{1}{3},\frac{2}{3}$ converges faster than the existing GD, normalized GD ($1$-GD) and clipped GD algorithms, which shows the advantage of using a proper normalization parameter $\beta$.

We further compare stochastic algorithms with fine-tuned learning rate $\gamma$ and fixed batch size $b=50$ over 500 iterations. 
This includes the basic SGD with $\gamma=2\times 10^{-4}$, normalized SGD with $\gamma=2\times 10^{-3}$, normalized SGD with momentum \citep{jin2021non} with $\gamma=3\times 10^{-3}$ and momentum coefficient $10^{-4}$, clipped SGD \cite{zhang2019gradient}  with $\gamma=0.3$ and normalization term $\max(\|\nabla f(z_t)\|,10^3)$, and SPIDER with $\gamma=0.01$, epoch size $q=5$ and batchsizes $B=3000, B'=50$. 
We generate the initialization by running $\frac{2}{3}$-GD with $\gamma=0.1$ for 100 iterations {from $z_0$}.
Figure \ref{figs} (top right) plots the comparison result on objective function value v.s. sample complexity. It can be seen that SPIDER uses slightly more samples at the beginning but converges to a much better solution than the other SGD-type algorithms. This demonstrates the advantage of applying both variance reduction and proper normalization to solve generalized-smooth nonconvex stochastic problems.

\subsection{Application to DRO}
In this section, we test our algorithms via solving the nonconvex DRO problem in \eqref{DRO2} on the life expectancy data\footnote{\url{https://www.kaggle.com/datasets/kumarajarshi/life-expectancy-who?resource=download}}, which collected the life expectancy (target) and its influencing factors (features) of 2413 people for regression analysis. We preprocess the data by filling the missing values with the median of the corresponding variables, censorizing and standardizing all the variables\footnote{The detailed process of filling missing values and censorization can be seen in \url{https://thecleverprogrammer.com/2021/01/06/life-expectancy-analysis-with-python/}}, removing two categorical variables (``country'' and ``status''), and adding standard Gaussian noise to the target to ensure model robustness. We select the first 2000 samples $\{x_i,y_i\}_{i=1}^{2000}$ as the training samples where $x_i\in \mathbb{R}^{34}$ and $y_i\in \mathbb{R}$ are feature and target respectively. In the DRO problem \eqref{DRO2}, we set $\lambda=0.01$ and select $\psi^*(t)=\frac{1}{4}(t+2)_+^2-1$ which corresponds to $\chi^2$ divergence. For any sample pair $x_{\xi},y_{\xi}$, we adopt the regularized mean square loss function $\ell_{\xi}(w)=\frac{1}{2}(y_{\xi}-x_{\xi}^{\top}w)^2+0.1\sum_{j=1}^{34}\ln\big(1+|w^{(j)}|)\big)$ with parameter $w=[w^{(1)};\ldots;w^{(34)}]\in\mathbb{R}^{34}$. We initialize $\eta_0=0.1$ and randomly initialize $w_0\in\mathbb{R}^{34}$ entrywise using standard Gaussian distribution. 

We first compare deterministic algorithms with fine-tuned learning rate $\gamma$ over 50 iterations. This includes the basic GD with $\gamma=10^{-4}$, clipped GD \cite{zhang2019gradient} with $\gamma=0.3$ and normalization term $\max(\|\nabla L(x_t,\eta_t)\|,10)$, and normalized GD (our $\beta$-GD with $\beta=1$) with $\gamma=0.2$, respectively. Figure \ref{figs} (bottom left) plots the comparison result on the objective function value $\Psi(x_t):=\min_{\eta\in\mathbb{R}} L(x_t,\eta)$ ($L$ is defined in eq. \eqref{DRO2}) v.s. iteration. It can be seen that  normalized GD and clipped GD converge to comparable function values and both outperform standard GD. 

We further compare stochastic algorithms with fine-tuned learning rate $\gamma$ and fixed minibatch size $b=50$ over 5000 iterations. This includes the basic SGD with $\gamma=2\times 10^{-4}$, normalized SGD with $\gamma=8\times 10^{-3}$, normalized SGD with momentum with $\gamma=8\times 10^{-3}$ and momentum coefficient $10^{-4}$, clipped SGD [1] with $\gamma=0.05$ and normalization term $\max(\|\nabla L(x_t,\eta_t)\|,100)$, and SPIDER with $\gamma=4\times 10^{-3}$, epoch size $q=20$ and batchsizes $B=2000, B'=50$. We generate the initialization by running normalized GD with $\gamma=0.2$ for 30 iterations from $w_0, \eta_0$. Figure \ref{figs} (bottom right) plots the comparison result on objective function value $\Psi(x_t)$ v.s. sample complexity. It can be seen that SPIDER takes slightly more samples at the beginning but converges to a better solution than the other SGD-type algorithms. This demonstrates the advantage of applying both variance reduction and proper normalization to solve generalized-smooth nonconvex stochastic problems.

\vspace{-2mm}
\section{Conclusion}
In this work, we proposed a new class of generalized-smooth functions that extends the existing ones. We developed both deterministic and stochastic gradient-based algorithms for solving problems in this class and obtained 
the optimal complexities. Our results extend the existing boundary of first-order nonconvex optimization and may inspire new developments in this direction. In the future, it is interesting to explore if other popular variance reduction algorithms such as STORM and SpiderBoost can be normalized to solve generalized-smooth nonconvex stochastic problems.

\section*{Acknowledgement}
The work of Ziyi and Yi Zhou was supported in part by U.S. National Science Foundation under the grants CCF-2106216, DMS-2134223 and CAREER-2237830.

The work of Yingbin was supported in part by the U.S. National Science Foundation under the grants CCF-1900145 and CCF-1909291.

The work of Zhaosong was supported in part by  U.S. National Science Foundation under the grant IIS-2211491.

\bibliographystyle{ieeetr}
\bibliography{Gmooth_Arxiv.bib}

\newpage
\onecolumn
\appendix

\addcontentsline{toc}{section}{Appendix} 
\part{Appendix} 
\parttoc 
\allowdisplaybreaks

\section{Supporting Lemmas}\label{supp:lemmas}
\begin{lemma}\label{lemma:int}
	$f\in\mathcal{L}_{\text{sym}}^*(\alpha)$ if and only if for any $w,w\in\mathbb{R}^d$,
	\begin{align}
		\|\nabla f(w')-\nabla f(w)\|\le \Big(L_0+L_1\int_0^1\|\nabla f(w_{\theta})\|^{\alpha}d\theta\Big)\|w'-w\|\label{int}
	\end{align}
	where $w_{\theta}:=\theta w'+(1-\theta)w$. 
\end{lemma}
Lemma \ref{lemma:int} provides an equivalent definition of $f\in\mathcal{L}_{\text{asym}}^*(1)$ which is sometimes more convenient to use than Definition \ref{def:P}, for example, in the proof in Section \ref{supp:proof_H}.

\begin{proof}
	Eq. \eqref{int} directly implies eq. \eqref{P} (i.e., $f\in\mathcal{L}_{\text{sym}}^*(\alpha)$) since
	\begin{align}
		\int_0^1\|\nabla f(w_{\theta})\|^{\alpha}d\theta\le \max_{\theta\in[0,1]}\|\nabla f(w_{\theta})\|^{\alpha}. \nonumber
	\end{align}
	Then it remains to prove eq. \eqref{int} given eq. \eqref{P}. 
	For any $n\in\mathbb{N}^+$, we have
	\begin{align}
		&\|\nabla f(w')-\nabla f(w)\|\nonumber\\
		&\le \sum_{k=0}^{n-1} \|\nabla f(w_{(k+1)/n})-\nabla f(w_{k/n})\|\nonumber\\
		&\overset{(i)}{\le} \sum_{k=0}^{n-1}\|w_{(k+1)/n}-w_{k/n}\| \Big(L_0+L_1\max_{\theta\in[0,1]}\|\nabla f(w_{\theta(k+1)/n+(1-\theta)k/n})\|^{\alpha}\Big)\nonumber\\
		&\overset{(ii)}{=} \|w'-w\| \sum_{k=0}^{n-1} \frac{1}{n}\max_{\theta\in[k/n,(k+1)/n]}h(\theta)\nonumber
	\end{align}
	where (i) uses eq. \eqref{P} with $w, w'$ replaced by $w_{k/n},w_{(k+1)/n}$ respectively ($w_{k/n}$ and $w_{(k+1)/n}$ denote $w_{\theta}$ with $\theta=k/n$ and $\theta=(k+1)/n$ respectively)  and (ii) denotes $h(\theta):=L_0+L_1\|\nabla f(w_{\theta})\|^{\alpha}$. Since $h(\cdot)$ is continuous, letting $n\to +\infty$ in the above inequality proves eq. \eqref{int} as follows.
	\begin{align}
		\|\nabla f(w')-\nabla f(w)\|\le \|w'-w\|\int_0^1 h(\theta)d\theta =\Big(L_0+L_1\int_0^1\|\nabla f(w_{\theta})\|^{\alpha}d\theta\Big)\|w'-w\|\nonumber
	\end{align}
\end{proof}

\begin{lemma}\label{lemma:avg_int}
	$f\in\mathbb{E}\mathcal{L}_{\text{sym}}^*(\alpha)$ if and only if for any $w,w\in\mathbb{R}^d$,
	\begin{align}
		\mathbb{E}_{\xi}\|\nabla f_{\xi}(w')-\nabla f_{\xi}(w)\|^2\le \|w'-w\|^2\mathbb{E}_{\xi}\int_{0}^1 \big(L_0 + L_1\|\nabla f_{\xi}(w_{\theta})\|^{\alpha}\big)^2d\theta\label{avg_int}
	\end{align}
	where $w_{\theta}:=\theta w'+(1-\theta)w$. 
\end{lemma}
Lemma \ref{lemma:avg_int} provides an equivalent definition of $f\in\mathbb{E}\mathcal{L}_{\text{asym}}^*(1)$ which is sometimes more convenient to use than Definition \ref{def:avgP}.
\begin{proof}
	It sufficies to prove the equivalence between eqs. \eqref{avg_int} \& \eqref{avgP}. 
	
	Eq. \eqref{avg_int} directly implies eq. \eqref{avgP} since 
	\begin{align}
		\int_0^1\!\big(L_0\!+\!L_1\|\nabla f_{\xi}(w_{\theta})\|^{\alpha}\big)^2d\theta\le \max_{\theta\in[0,1]}\big(L_0\!+\!L_1\|\nabla f_{\xi}(w_{\theta})\|^{\alpha}\big)^2= \big(L_0\!+\!L_1\max_{\theta\in[0,1]}\|\nabla f_{\xi}(w_{\theta})\|^{\alpha}\big)^2.\nonumber
	\end{align}
	Then it remains to prove eq. \eqref{avg_int} given eq. \eqref{avgP}. For any $w,w'\in\mathbb{R}^d$, denote $w_{\theta}:=\theta w'+(1-\theta)w$. Then for any $\theta\in[0,1]$ and $n\in\mathbb{N}^+$, we have
	\begin{align}
		&\mathbb{E}_{\xi}\|\nabla f_{\xi}(w_{\theta})-\nabla f_{\xi}(w)\|^2\nonumber\\
		&=\mathbb{E}_{\xi}\Big\|\sum_{k=0}^{n-1} \big(\nabla f_{\xi}(w_{\theta(k+1)/n})-\nabla f_{\xi}(w_{\theta k/n})\big)\Big\|^2\nonumber\\
		&\overset{(i)}{\le} n\sum_{k=0}^{n-1}\mathbb{E}_{\xi}\big\|\nabla f_{\xi}(w_{\theta(k+1)/n})-\nabla f_{\xi}(w_{\theta k/n})\big\|^2\nonumber\\
		&\overset{(ii)}{\le}n \sum_{k=0}^{n-1} \|w_{\theta(k+1)/n}-w_{\theta k/n}\|^2\mathbb{E}_{\xi}\Big(L_0+L_1\max_{\theta'\in[0,1]}\|\nabla f_{\xi}\big(\theta'w_{\theta(k+1)/n}+(1-\theta')w_{\theta k/n}\big)\|^{\alpha}\Big)^2\nonumber\\
		&=\theta^2\|w'-w\|^2\mathbb{E}_{\xi}\sum_{k=0}^{n-1} \frac{1}{n}\max_{\theta'\in[0,1]}\big(L_0+L_1\|\nabla f_{\xi}\big(w_{\theta'\theta(k+1)/n+(1-\theta')\theta k/n}\big)\|^{\alpha}\big)^2\nonumber\\
		&\overset{(iii)}{=} \theta^2\|w'-w\|^2\mathbb{E}_{\xi}\sum_{k=0}^{n-1} \frac{1}{n}\max_{u\in[k/n,(k+1)/n]}h(u)\nonumber
	\end{align}
	where (i) applies Jensen' inequality to the convex function $\|\cdot\|^2$, (ii) uses eq. \eqref{avgP}, and (iii) denotes $h(u):=\big(L_0+L_1\|\nabla f_{\xi}(w_{\theta u})\|^{\alpha}\big)^2$. Since $h$ is a continuous function, letting $n\to +\infty$ in the above inequality yields that
	\begin{align}
		\mathbb{E}_{\xi}\|\nabla f_{\xi}(w_{\theta})-\nabla f_{\xi}(w)\|^2&\le \theta^2\|w'-w\|^2\mathbb{E}_{\xi}\int_0^1 h(u)du \le \theta^2\|w'-w\|^2\mathbb{E}_{\xi}\int_0^1 \big(L_0+L_1\|\nabla f_{\xi}(w_{\theta u})\|^{\alpha}\big)^2du. \label{theta_int_avg}
	\end{align}
	Substituting $\theta=1$ into the above inequality proves eq. \eqref{avg_int}. 
\end{proof}


\begin{lemma}\label{lemma:moment}
	Under Assumption \ref{assum:var}, the stochastic gradient $\nabla f_{\xi}(w)$ and true gradient $\nabla f(w)$ satisfy the following inequalities for any $\tau\in[0,2]$,
	\begin{align}
		&\mathbb{E}_{\xi\sim\mathbb{P}}\|\nabla f_{\xi}(w)\|^{\tau}\le (\Gamma^{\tau}+1)\|\nabla f(w)\|^{\tau}+\Lambda^{\tau}.\label{ESG2}
	\end{align}
\end{lemma}
\begin{proof}
	First, when $\tau=2$, Assumption \ref{assum:var} implies eq. \eqref{ESG2} as follows. 
	\begin{align}
		\mathbb{E}_{\xi\sim\mathbb{P}}\|\nabla f_{\xi}(w)\|^2\overset{(i)}{=}\mathbb{E}_{\xi\sim\mathbb{P}}\|\nabla f_{\xi}(w)-\nabla f(w)\|^2+\|\nabla f(w)\|^2\le (\Gamma^{2}+1)\|\nabla f(w)\|^{2}+\Lambda^{2} \label{ESG3}
	\end{align}
	where (i) uses $f(w)=\mathbb{E}_{\xi} f_{\xi}$. Then, when $\tau\in[0,2)$, we prove  eq. \eqref{ESG2} as follows. 
	\begin{align}
		\mathbb{E}_{\xi\sim\mathbb{P}}\|\nabla f_{\xi}(w)\|^{\tau}&\overset{(i)}{\le}(\mathbb{E}_{\xi\sim\mathbb{P}}\|\nabla f_{\xi}(w)\|^2)^{\tau/2} \overset{(ii)}{\le}\big((\Gamma^2+1)\|\nabla f(w)\|^2+\Lambda^2\big)^{\tau/2}\overset{(iii)}{=}(\Gamma^{\tau}+1)\|\nabla f(w)\|^{\tau}+\Lambda^{\tau}, \nonumber
	\end{align}
	where (i) applies Jensen's inequality to the concave function $g(s)=s^{\tau/2}$, (ii) uses eq. \eqref{ESG3}, and (iii) uses the inequality that $(a+b)^{\tau/2}\le a^{\tau/2}+b^{\tau/2}$ for any $a,b\ge 0$ and $\tau/2\in[0,1]$. 
\end{proof}

Note that the only randomness of Algorithm \ref{algo: SPIDER} comes from  $S_t$, so we can consider the filtration $\mathcal{F}(S_{1:t}):=\mathcal{F}(S_1,\ldots,S_t)$ which monotonically increases with larger $t$. Then, it can be easily seen from Algorithm \ref{algo: SPIDER} that 
\begin{align}
	v_t,w_{t+1},\delta_{t+1}\in \mathcal{F}(S_{1:t}) / \mathcal{F}(S_{1:(t-1)})\label{filter}
\end{align}

\section{Proof of Theorem \ref{thm: func_class}}\label{supp: proof_func_class}
\subsection{Proof of Item \ref{item:asym}}
On one hand, $f\in\mathbf{\mathcal{L}}_{\text{asym}}^*$ means $\|\nabla f(w')-\nabla f(w)\|\le \big(L_0+L_1\|\nabla f(w')\|\big) \|w'-w\|$ for all $w,w'$, which directly implies eq. \eqref{P} with $\alpha=1$, i.e., $f\in\mathcal{L}_{\text{asym}}^*(1)$. On the other hand, we will prove item \ref{item:sym_exp} which shows that $f(w)=e^w+e^{-w},w\in\mathbb{R}$ belongs to $\mathcal{L}_{\text{sym}}^*(1)$ but not $\mathbf{\mathcal{L}}_{\text{asym}}^*$. Therefore, $\mathcal{L}_{\text{asym}}^* \subset \mathcal{L}_{\text{sym}}^*(1)$. 

\subsection{Proof of Item \ref{item:H}}\label{supp:proof_H}
Note that if a function is not twice-differentiable, it cannot belong to $\mathcal{L}_{\text{H}}^*$ but may still belong to $\mathcal{L}_{\text{sym}}^*(1)$. For example, for the function $f(w)=w|w|$ 
whose derivative $f'(w)=2|w|$ is not differentiable (so $f\notin\mathcal{L}_{\text{H}}^*$), we have $f\in\mathcal{L}\subset\mathcal{L}_{\text{sym}}^*(1)$ since $|f'(w')-f'(w)|\le 2\big||w'|-|w|\big|\le 2|w'-w|$. 

Therefore, it remains to prove for twice-differentiable functions $f$ the equivalence between 
eq. \eqref{df2} below (definition of $\mathcal{L}_{\text{H}}^*$) and eq. \eqref{int} with $\alpha=1$ (equivalent definition of $\mathcal{L}_{\text{sym}}^*(1)$).
\begin{align}
	\|\nabla^2 f(w')\|\le L_0+L_1\|\nabla f(w)\|.\label{df2}
\end{align}

Eq. \eqref{df2} implies eq. \eqref{int} as proved below.
\begin{align}
	\|\nabla f(w')-\nabla f(w)\|&=\Big\|\int_{0}^1 \nabla^2 f(w_{\theta})(w'-w)d\theta\Big\|\nonumber\\
	&\le \int_{0}^1 \|\nabla^2 f(w_{\theta})\|\|w'-w\|d\theta\nonumber\\
	&\overset{(i)}{\le}\|w'-w\|\int_{0}^1 \big(L_0 + L_1\|\nabla f(w_{\theta})\|^{\alpha}\big)d\theta \nonumber\\
	&=\|w'-w\|\!\Big(\!L_0+L_1\int_0^1\|\nabla f(w_{\theta})\|^{\alpha}d\theta\!\Big),\nonumber
\end{align}
where (i) uses eq. \eqref{df2}. Finally, it remains prove eq. \eqref{df2} given eq. \eqref{int}. 

Note that of the symmetric Hessian matrix $\nabla^2 f(w)$ has eigenvalue $\|\nabla^2 f(w)\|$ or $-\|\nabla^2 f(w)\|$. Denote $s$ as the corresponding eigenvector with $\|s\|=1$, i.e., $\nabla^2 f(w)s=\pm\|\nabla^2 f(w)\|s$. In eq. \eqref{int} , we adopt $w':=w+\theta' s$ ($\theta'\in(0,1)$), so $w_{\theta}:=\theta w'+(1-\theta)w=\theta (w+\theta' s)+(1-\theta)w=w+\theta\theta's$ and thus eq. \eqref{int} becomes
\begin{align}
	\|\nabla f(w+\theta's)-\nabla f(w)\|&\le \!\theta'\Big(\!L_0+L_1\int_0^1\|\nabla f(w+\theta\theta' s)\|^{\alpha}d\theta\!\Big)\! \label{df_mid}
\end{align}
The left side of eq. \eqref{df_mid} can be rewritten as follows.
\begin{align}
	\|\nabla f(w+\theta's)-\nabla f(w)\|&=\theta'\Big\|\int_{0}^1 \nabla^2 f\big(\theta (w+\theta's)+(1-\theta)w\big)d\theta\Big\|\nonumber\\
	&=\Big\|\int_{0}^1 \nabla^2 f(w+\theta\theta's)\theta'd\theta\Big\|\nonumber\\
	&\overset{(i)}{=}\Big\|\int_0^{\theta'}\nabla^2 f(w+us)sdu\Big\|,\label{df_left}
\end{align}
where (i) uses change of variables $u=\theta'\theta$. The right side of eq. \eqref{df_mid} can be rewritten as follows.
\begin{align}
	&\!\theta'\Big(\!L_0+L_1\int_0^1\|\nabla f(w+\theta\theta' s)\|^{\alpha}d\theta\!\Big)\overset{(i)}{=}L_0\theta'+L_1\int_0^{\theta'}\|\nabla f(w+us)\|^{\alpha}du,\label{df_right}
\end{align}
where (i) also uses change of variables $u=\theta'\theta$. Substituting eqs. \eqref{df_left} \& \eqref{df_right} into eq. \eqref{df_mid} and multiplying both sides by $1/\theta'>0$, we obtain that
\begin{align}
	\Big\|\frac{1}{\theta'}\int_0^{\theta'}\nabla^2 f(w+us)sdu\Big\|\le L_0+\frac{L_1}{\theta'}\int_0^{\theta'}\|\nabla f(w+us)\|^{\alpha}du.\nonumber
\end{align}
Letting $\theta'\to+0$ in the above inequality, we obtain eq. \eqref{df2} as follows.
\begin{align}
	\|\nabla^2 f(w)\|=\|\nabla^2 f(w)s\|\le L_0+L_1\|\nabla f(w)\|^{\alpha}.\nonumber
\end{align}

\subsection{Proof of Item \ref{item:sym_poly}}
The polynomial function $f(w)=|w|^{\frac{2-\alpha}{1-\alpha}}, w\in\mathbb{R}$ is twice-differentiable with first and second order derivatives below.
\begin{align}
	f'(w)&=\frac{2-\alpha}{1-\alpha}|w|^{\frac{1}{1-\alpha}}\text{sgn}(w), \quad f''(w)=\frac{2-\alpha}{(1-\alpha)^2}|w|^{\frac{\alpha}{1-\alpha}}.\nonumber
\end{align}
Therefore, for any $w,w'\in\mathbb{R}$, we have
\begin{align}
	|f'(w')-f'(w)|&\le |w'-w|\max_{\theta\in[0,1]} |f''(w_{\theta})|\nonumber\\
	&\le \frac{2-\alpha}{(1-\alpha)^2}|w'-w|\max_{\theta\in[0,1]}|w_{\theta}|^{\frac{\alpha}{1-\alpha}}\cdot 1 \nonumber\\
	&=\frac{2-\alpha}{(1-\alpha)^2}|w'-w|\max_{\theta\in[0,1]}\Big|\frac{1-\alpha}{2-\alpha}f'(w_{\theta})\Big|^{\alpha} \nonumber\\
	&\le \frac{(2-\alpha)^{1-\alpha}}{(1-\alpha)^{2-\alpha}}|w'-w|\max_{\theta\in[0,1]}|f'(w_{\theta})|^{\alpha}
\end{align}
where $w_{\theta}:=\theta w'+(1-\theta)w$. This verifies eq. \eqref{P} and thus proves that $f\in\mathcal{L}_{\text{sym}}(\alpha)$. 

Next, we prove that $f\not\in \mathcal{L}_{\text{sym}}^*(\widetilde{\alpha})$ for all $\widetilde{\alpha} \in (0,\alpha)$. Suppose $f\in\mathcal{L}_{\text{sym}}^*(\widetilde{\alpha})$, i.e., the following inequality holds for all $w,w'\in\mathbb{R}^d$. 
\begin{align}
	|f'(w')-f'(w)|&\le |w'-w|\big(L_0+L_1\max_{\theta\in[0,1]} |f'(w_{\theta})|^{\widetilde{\alpha}}\big),\nonumber
\end{align}
where $w_{\theta}:=\theta w'+(1-\theta)w$. Substituting $w'=0$ into the above inequality, we obtain the following inequality for all $w\in\mathbb{R}^d$.
\begin{align}
	\frac{2-\alpha}{1-\alpha}|w|^{\frac{1}{1-\alpha}}\le |w|\Big(L_0+L_1\Big(\frac{2-\alpha}{1-\alpha}|w|^{\frac{1}{1-\alpha}}\Big)^{\widetilde{\alpha}}\Big).\nonumber
\end{align}
As $|w|\to+\infty$, the left side of the above inequality is $\mathcal{O}(|w|^{\frac{1}{1-\alpha}})$ whereas the right side has strictly smaller order $\mathcal{O}(|w|^{\frac{1-\alpha+\widetilde{\alpha}}{1-\alpha}})$. Hence, the above inequality cannot hold for sufficiently large $|w|$, which means the assumption that $f\in\mathcal{L}_{\text{sym}}^*(\widetilde{\alpha})$ does not hold. 

Finally, we prove that $f\notin\mathcal{L}_{\text{asym}}^*$. Suppose  $f\in\mathcal{L}_{\text{asym}}^*$, i.e., the following inequality holds for all $w,w'\in\mathbb{R}^d$. 
\begin{align}
	|f'(w')-f'(w)|&\le |w'-w|\big(L_0+L_1|f'(w)|\big).\nonumber
\end{align}
Substituting $w=0$ into the above inequality, we obtain the following inequality for all $w'\in\mathbb{R}^d$.
\begin{align}
	\frac{2-\alpha}{1-\alpha}|w'|^{\frac{1}{1-\alpha}}\le L_0|w'|,\nonumber
\end{align}
which implies that $|w'|\le \big(\frac{L_0(1-\alpha)}{2-\alpha}\big)^{\frac{1-\alpha}{\alpha}}<+\infty$. Hence, the above inequality cannot for all sufficiently large $|w'|$,  which means the assumption that $f\in\mathcal{L}_{\text{asym}}^*$ does not hold.

\subsection{Proof of Item \ref{item:sym_exp}}

The exponential function $f(w)=e^{w}+e^{-w}, w\in\mathbb{R}$ is twice-differentiable with first and second order derivatives below.
\begin{align}
	f'(w)&=e^{w}-e^{-w}=\text{sgn}(w)\big(e^{|w|}-e^{-|w|}\big), \quad f''(w)=e^{w}+e^{-w}=e^{|w|}+e^{-|w|}.\nonumber
\end{align}

When $|w|\le 1$, $|f''(w)|\le e+e^{-1}<4$; When $|w|>1$, $|f'(w)|+4=e^{|w|}+e^{-|w|}-2e^{-|w|}+4>|f''(w)|-2e^{-1}+4>|f''(w)|$. Combining the two cases yields that $|f''(w)|<|f'(w)|+4$, which implies that $f\in\mathcal{L}_{\text{sym}}^*$. Since $f$ is twice-differentiable, we have $f\in\mathcal{L}_{\text{sym}}^*(1)$ based on item \ref{item:H} of Theorem \ref{thm: func_class}. 

Next, we prove that $f\not\in \mathcal{L}_{\text{sym}}^*(\widetilde{\alpha})$ for all $\widetilde{\alpha} \in (0,1)$. Suppose $f\in\mathcal{L}_{\text{sym}}^*(\widetilde{\alpha})$, i.e., the following inequality holds for all $w,w'\in\mathbb{R}^d$. 
\begin{align}
	|f'(w')-f'(w)|&\le |w'-w|\big(L_0+L_1\max_{\theta\in[0,1]} |f'(w_{\theta})|^{\widetilde{\alpha}}\big),\nonumber
\end{align}
where $w_{\theta}:=\theta w'+(1-\theta)w$. Substituting $w'=0$ into the above inequality, we obtain the following inequality.
\begin{align}
	e^{|w|}-e^{-|w|}\le|w|\big(L_0+L_1(e^{|w|}-e^{-|w|})^{\widetilde{\alpha}}\big),\forall w\in\mathbb{R}^d,\nonumber
\end{align}
which implies that
\begin{align}
	\frac{(e^{|w|}-e^{-|w|})^{1-\widetilde{\alpha}}}{|w|}\le\frac{L_0+L_1(e^{|w|}-e^{-|w|})^{\widetilde{\alpha}}}{(e^{|w|}-e^{-|w|})^{\widetilde{\alpha}}},\forall w\in\mathbb{R}^d/\{\zero\}.\nonumber
\end{align}
As $|w|\to+\infty$, the left side of the above inequality goes to $+\infty$ while the right sides converges to $L_1<+\infty$. Hence, the above inequality cannot hold for sufficiently large $|w|$, which means the assumption that $f\in\mathcal{L}_{\text{sym}}^*(\widetilde{\alpha})$ does not hold. 

Finally, we prove that $f\notin\mathcal{L}_{\text{asym}}^*$. Suppose  $f\in\mathcal{L}_{\text{asym}}^*$, i.e., the following inequality holds for all $w,w'\in\mathbb{R}^d$. 
\begin{align}
	|f'(w')-f'(w)|&\le |w'-w|\big(L_0+L_1|f'(w)|\big).\nonumber
\end{align}
Substituting $w=0$ into the above inequality and rearranging it, we obtain the following inequality for all $w'\in\mathbb{R}^d/\{\zero\}$.
\begin{align}
	\frac{e^{|w'|}-e^{-|w'|}}{|w'|}\le L_0,\nonumber
\end{align}
As $|w|\to+\infty$, the left side of the above inequality goes to $+\infty$, so the above inequality cannot for all sufficiently large $|w'|$, which means the assumption that $f\in\mathcal{L}_{\text{asym}}^*$ does not hold. 

\section{Proof of Proposition \ref{prop:PS_equiv}}

\subsection{Proof of Item \ref{item:df_bound_poly}}
First, we prove eq. \eqref{PS_poly_up} for $f\in\mathcal{L}_{\text{sym}}^*(\alpha)$ with $\alpha\in(0,1)$. Note that eq. \eqref{int} holds for all $w,w'\in\mathbb{R}^d$. Hence, for any $\theta'\in[0,1]$, we can replace $w'$ with $w_{\theta'}:=\theta' w'+(1-\theta')w$ in eq. \eqref{int}, so $w_{\theta}$ becomes $\theta' w_{\theta}+(1-\theta)w=\theta'\theta w'+(1-\theta'\theta)w=w_{\theta'\theta}$. Therefore, eq. \eqref{int} becomes 
\begin{align}
	\|\nabla f(w_{\theta'})-\nabla f(w)\|&\le \Big(L_0+L_1\int_0^1\|\nabla f(w_{\theta'\theta})\|^{\alpha}d\theta\Big)\|w_{\theta'}-w\|\nonumber\\
	&=\Big(L_0\theta'+L_1\int_0^1\|\nabla f(w_{\theta'\theta})\|^{\alpha}\theta'd\theta\Big)\|w'-w\|\nonumber\\
	&\overset{(i)}{=}H(\theta')\|w'-w\|\label{theta_int2}
\end{align}
where (i) denotes $H(\theta'):=L_0\theta'+L_1\int_0^1\|\nabla f(w_{\theta'\theta})\|^{\alpha}\theta'd\theta=L_0\theta'+L_1\int_0^{\theta'}\|\nabla f(w_u)\|^{\alpha}du$. Then its derivative $H'(\theta)$ can be bounded as follows,
\begin{align}
	H'(\theta')&=L_0+L_1\|\nabla f(w_{\theta'})\|^{\alpha}\nonumber\\
	&\le L_0+L_1\|\nabla f(w_{\theta'})-\nabla f(w)\|^{\alpha}+L_1\|\nabla f(w)\|^{\alpha}\nonumber\\
	&\overset{(i)}{\le} L_0+L_1\|w'-w\|^{\alpha}H(\theta')^{\alpha}+L_1\|\nabla f(w)\|^{\alpha}\label{H_ODE}\\
	&\overset{(ii)}{\le} 3L_1\Big(\frac{1}{3}\|w'-w\|H(\theta')+\frac{1}{3}\|\nabla f(w)\|+\frac{L_0^{\frac{1}{\alpha}}}{3L_1^{\frac{1}{\alpha}}}\Big)^{\alpha}.\nonumber
\end{align}
where (i) uses eq. \eqref{theta_int2} and (ii) applies Jensen's inequality to the concave function $g(x)=x^{\alpha}$. Rearranging the above inequality yields that
\begin{align}
	3^{1-\alpha}L_1(1-\alpha)\|w'-w\|&\ge (1-\alpha)\|w'-w\|\Big(\|w'-w\|H(\theta')+\|\nabla f(w)\|+\frac{L_0^{\frac{1}{\alpha}}}{L_1^{\frac{1}{\alpha}}}\Big)^{-\alpha}H'(\theta')\nonumber\\
	&=\frac{d}{d\theta'}\Big(\|w'-w\|H(\theta')+\|\nabla f(w)\|+\frac{L_0^{\frac{1}{\alpha}}}{L_1^{\frac{1}{\alpha}}}\Big)^{1-\alpha}.\nonumber
\end{align}
Integrating the above inequality over $\theta'\in[0,\theta]$ yields that
\begin{align}
	&\Big(\|w'-w\|H(\theta)+\|\nabla f(w)\|+\frac{L_0^{\frac{1}{\alpha}}}{L_1^{\frac{1}{\alpha}}}\Big)^{1-\alpha}\nonumber\\
	&\le 3^{1-\alpha}L_1(1-\alpha)\|w'-w\|\theta+\Big(\|w'-w\|H(0)+\|\nabla f(w)\|+\frac{L_0^{\frac{1}{\alpha}}}{L_1^{\frac{1}{\alpha}}}\Big)^{1-\alpha}\nonumber\\
	&\overset{(i)}{\le} 2^{\alpha}\Big(3\big(L_1(1-\alpha)\|w'-w\|\theta\big)^{\frac{1}{1-\alpha}}+\|\nabla f(w)\|+\frac{L_0^{\frac{1}{\alpha}}}{L_1^{\frac{1}{\alpha}}}\Big)^{1-\alpha}\nonumber
\end{align}
where (i) uses $H(0)=0$ and applies Jensen's inequality to the concave function $g(x)=x^{1-\alpha}$. Therefore,
\begin{align}
	\|w'-w\|H(\theta)\le 2^{\frac{\alpha}{1-\alpha}}\Big(3\big(L_1(1-\alpha)\|w'-w\|\theta\big)^{\frac{1}{1-\alpha}}+\|\nabla f(w)\|+\frac{L_0^{\frac{1}{\alpha}}}{L_1^{\frac{1}{\alpha}}}\Big)-\|\nabla f(w)\|-\frac{L_0^{\frac{1}{\alpha}}}{L_1^{\frac{1}{\alpha}}}.\nonumber
\end{align}
Substituting the above inequality into eq. \eqref{theta_int2}, we obtain that 
\begin{align}
	\|\nabla f(w_{\theta})\|&\le \|\nabla f(w)\|+\|\nabla f(w_{\theta})-\nabla f(w)\|\nonumber\\
	&\le \|\nabla f(w)\|+\|w'-w\|H(\theta)\nonumber\\
	&\le 2^{\frac{\alpha}{1-\alpha}}\Big(3\big(L_1(1-\alpha)\|w'-w\|\theta\big)^{\frac{1}{1-\alpha}}+\|\nabla f(w)\|+\frac{L_0^{\frac{1}{\alpha}}}{L_1^{\frac{1}{\alpha}}}\Big).\nonumber
\end{align}
Then, substituting the above inequality into eq. \eqref{P}, we obtain that
\begin{align}
	&\|\nabla f(w')-\nabla f(w)\|\nonumber\\
	&\le \big(L_0+L_1\max_{\theta\in[0,1]}\|\nabla f(w_{\theta})\|^{\alpha}\big) \|w'-w\| \nonumber\\
	&\le\Big(L_0+L_1\cdot 2^{\frac{\alpha^2}{1-\alpha}}\Big(3\big(L_1(1-\alpha)\|w'-w\|\big)^{\frac{1}{1-\alpha}}+\|\nabla f(w)\|+\frac{L_0^{\frac{1}{\alpha}}}{L_1^{\frac{1}{\alpha}}}\Big)^{\alpha}\Big)\|w'-w\|  \nonumber\\
	&\overset{(i)}{\le}\Big(L_0+L_1\cdot 2^{\frac{\alpha^2}{1-\alpha}}\Big(3^{\alpha}\big(L_1(1-\alpha)\|w'-w\|\big)^{\frac{\alpha}{1-\alpha}}+\|\nabla f(w)\|^{\alpha}+\frac{L_0}{L_1}\Big)\Big)\|w'-w\|  \nonumber\\
	&=\|w'-w\|\big(K_0+K_1\|\nabla f(w)\|^{\alpha}\!+\!K_2\|w'-w\|^{\frac{\alpha}{1-\alpha}}\big)\nonumber
\end{align}
where (i) uses the inequality that $(a+b+c)^{\alpha}\le a^{\alpha}+b^{\alpha}+c^{\alpha}$ for any $a,b,c\ge 0$ and $\alpha\in[0,1]$, and (ii) denotes that $K_0:=L_0\big(2^{\frac{\alpha^2}{1-\alpha}}+1\big)$, $K_1:=L_1\cdot 2^{\frac{\alpha^2}{1-\alpha}}\cdot 3^{\alpha}$, $K_2:=L_1^{\frac{1}{1-\alpha}}\cdot 2^{\frac{\alpha^2}{1-\alpha}}\cdot 3^{\alpha}(1-\alpha)^{\frac{\alpha}{1-\alpha}}$. 

Next, we prove $f\in\mathcal{L}_{\text{sym}}^*(\alpha)$ given eq. \eqref{PS_poly_up}. For any $w,w'\in\mathbb{R}^d$ and $n\in\mathbb{N}^+$, we have
\begin{align}
	&\|\nabla f(w')-\nabla f(w)\|\nonumber\\
	&\overset{(i)}{\le} \sum_{k=0}^{n-1} \|\nabla f(w_{(k+1)/n})-\nabla f(w_{k/n})\|\nonumber\\
	&\overset{(ii)}{\le} \sum_{k=0}^{n-1}\|w_{(k+1)/n}-w_{k/n}\| \big(K_0+K_1\|\nabla f(w_{k/n})\|^{\alpha}\!+\!K_2\|w_{(k+1)/n}-w_{k/n}\|^{\frac{\alpha}{1-\alpha}}\big)\nonumber\\
	&\overset{(iii)}{=} \|w'-w\| \sum_{k=0}^{n-1} \Big(\frac{1}{n}h\Big(\frac{k}{n}\Big)+\frac{1}{n}\cdot \frac{K_2}{n^{\frac{\alpha}{1-\alpha}}}\|w'-w\|^{\frac{\alpha}{1-\alpha}}\Big)\nonumber\\
	&=\|w'-w\| \Big(\frac{K_2}{n^{\frac{\alpha}{1-\alpha}}}\|w'-w\|^{\frac{\alpha}{1-\alpha}}+\sum_{k=0}^{n-1} \frac{1}{n}h\Big(\frac{k}{n}\Big)\Big),\nonumber
\end{align}
where (i) denotes $w_{\theta}:=\theta w'+(1-\theta)w$, (ii) uses eq. \eqref{PS_poly_up} with $w, w'$ replaced by $w_{k/n},w_{(k+1)/n}$ respectively and (iii) denotes $h(\theta):=K_0+K_1\|\nabla f(w_{\theta})\|^{\alpha}$. Since $h(\cdot)$ is continuous, letting $n\to +\infty$ in the above inequality proves eq. \eqref{int} as follows, which implies $f\in\mathcal{L}_{\text{sym}}^*(\alpha)$ by Lemma \ref{lemma:int}. 
\begin{align}
	&\|\nabla f(w')-\nabla f(w)\|\le \|w'-w\|\int_0^1h(\theta)d\theta=\Big(L_0+L_1\int_0^1\|\nabla f(w_{\theta})\|^{\alpha}d\theta\Big)\|w'-w\|\nonumber
\end{align}

\subsection{Proof of Item \ref{item:df_bound_exp}}
Note that eq. \eqref{H_ODE} holds for any function $f\in\mathcal{L}_{\text{sym}}^*(\alpha)$ with $\alpha\in[0,1]$. Substituting $\alpha=1$ into eq. \eqref{H_ODE}, we obtain that
\begin{align}
	H'(\theta)&\le L_0+L_1\|w'-w\|H(\theta)+L_1\|\nabla f(w)\|,\nonumber
\end{align}
where $H(\theta'):=L_0\theta'+L_1\int_0^{\theta'}\|\nabla f(w_u)\|du$. Rearranging the above inequality yields that
\begin{align}
	L_1\|w'-w\|&\ge \frac{L_1\|w'-w\|H'(\theta')}{L_0+L_1\|w'-w\|H(\theta')+L_1\|\nabla f(w)\|}= \frac{d}{d\theta'} \ln\big(L_0+L_1\|w'-w\|H(\theta')+L_1\|\nabla f(w)\|\big).\nonumber
\end{align}
Integrating the above inequality over $\theta'\in[0,\theta]$ yields that (note that $H(0)=0$)
\begin{align}
	\ln\big(L_0+L_1\|w'-w\|H(\theta)+L_1\|\nabla f(w)\|\big)&\le \ln\big(L_0+L_1\|\nabla f(w)\|\big)+L_1\|w'-w\|,\nonumber
\end{align}
which implies that
\begin{align}
	L_1\|w'-w\|H(\theta)\le \big(L_0+L_1\|\nabla f(w)\|\big)\exp\big(L_1\|w'-w\|\big)-L_0-L_1\|\nabla f(w)\|. \nonumber
\end{align}
Substituting the above inequality and $\alpha=1$ into eq. \eqref{theta_int2}, we obtain that 
\begin{align}
	\|\nabla f(w_{\theta})\|&\le \|\nabla f(w)\|+\|\nabla f(w_{\theta})-\nabla f(w)\|\nonumber\\
	&\le \|\nabla f(w)\|+\|w'-w\|H(\theta)\nonumber\\
	&\le \|\nabla f(w)\|+\frac{1}{L_1}\Big(\big(L_0+L_1\|\nabla f(w)\|\big)\exp\big(L_1\|w'-w\|\big)-L_0-L_1\|\nabla f(w)\|\Big)\nonumber\\
	&=\Big(\frac{L_0}{L_1}+\|\nabla f(w)\|\Big)\exp\big(L_1\|w'-w\|\big)-\frac{L_0}{L_1}\nonumber
\end{align}
Then, substituting the above inequality and $\alpha=1$ into eq. \eqref{P}, we prove eq. \eqref{PS_exp_up} as follows.
\begin{align}
	&\|\nabla f(w')-\nabla f(w)\|\nonumber\\
	&\le \big(L_0+L_1\max_{\theta\in[0,1]}\|\nabla f(w_{\theta})\|\big) \|w'-w\| \nonumber\\
	&\le\big(L_0+L_1\|\nabla f(w)\|\big)\exp\big(L_1\|w'-w\|\big)\|w'-w\|  \nonumber
\end{align}


Next, we prove $f\in\mathcal{L}_{\text{sym}}^*(\alpha)$ given eq. \eqref{PS_exp_up}. For any $w,w'\in\mathbb{R}^d$ and $n\in\mathbb{N}^+$, we have
\begin{align}
	&\|\nabla f(w')-\nabla f(w)\|\nonumber\\
	&\overset{(i)}{\le} \sum_{k=0}^{n-1} \|\nabla f(w_{(k+1)/n})-\nabla f(w_{k/n})\|\nonumber\\
	&\overset{(ii)}{\le} \sum_{k=0}^{n-1}\|w_{(k+1)/n}-w_{k/n}\|\big(L_0+L_1\|\nabla f(w_{k/n})\|\big)\exp\big(L_1\|w_{(k+1)/n}-w_{k/n}\|\big)\nonumber\\
	&\overset{(iii)}{=} \|w'-w\| \sum_{k=0}^{n-1} \frac{1}{n}h\Big(\frac{k}{n}\Big)\exp\Big(\frac{L_1}{n}\|w'-w\|\Big)\nonumber\\
	&= \|w'-w\| \sum_{k=0}^{n-1} \frac{1}{n}h\Big(\frac{k}{n}\Big)+\|w'-w\| \sum_{k=0}^{n-1} \frac{1}{n}h\Big(\frac{k}{n}\Big)\Big[\exp\Big(\frac{L_1}{n}\|w'-w\|\Big)-1\Big]\nonumber\\
	&\le \|w'-w\| \sum_{k=0}^{n-1} \frac{1}{n}h\Big(\frac{k}{n}\Big)+\|w'-w\| \max_{\theta\in[0,1]}h(\theta)\Big[\exp\Big(\frac{L_1}{n}\|w'-w\|\Big)-1\Big]\nonumber
\end{align}
where (i) denotes $w_{\theta}:=\theta w'+(1-\theta)w$, (ii) uses eq. \eqref{PS_exp_up} with $w, w'$ replaced by $w_{k/n},w_{(k+1)/n}$ respectively and (iii) denotes $h(\theta):=L_0+L_1\|\nabla f(w_{\theta})\|$. Since $h(\cdot)$ is continuous, letting $n\to +\infty$ in the above inequality proves eq. \eqref{int} with $\alpha=1$ as follows, which implies $f\in\mathcal{L}_{\text{sym}}^*(1)$ by Lemma \ref{lemma:int}. 
\begin{align}
	&\|\nabla f(w')-\nabla f(w)\|\le \|w'-w\|\int_0^1h(\theta)d\theta=\Big(L_0+L_1\int_0^1\|\nabla f(w_{\theta})\|d\theta\Big)\|w'-w\|\nonumber
\end{align}

\subsection{Proof of Item \ref{item:fup_poly}}
Since $f\in \mathcal{L}_{\text{sym}}^*(\alpha)$ for $\alpha\in (0,1)$, eq. \eqref{PS_poly_up} holds based on item \ref{item:df_bound_poly} of Proposition \ref{prop:PS_equiv}. Hence, for any $w,w'\in\mathbb{R}^d$, we prove eq. \eqref{PS_poly_fup} as follows.
\begin{align}
	&f(w')-f(w)-\nabla f(w)^{\top}(w'-w)\nonumber\\
	&=\int_0^1 \big(\nabla f(w_{\theta})-\nabla f(w)\big)^{\top}(w'-w) d\theta\nonumber\\
	&\le\int_0^1 \|\nabla f(w_{\theta})-\nabla f(w)\| \|w'-w\| d\theta\nonumber\\
	&\overset{(i)}{\le}\int_0^1 \|w_{\theta}-w\|\big(K_0+K_1\|\nabla f(w)\|^{\alpha}\!+\!K_2\|w_{\theta}-w\|^{\frac{\alpha}{1-\alpha}}\big) \|w'-w\| d\theta\nonumber\\
	&=\int_0^1 \theta\|w'-w\|^2\big(K_0+K_1\|\nabla f(w)\|^{\alpha}\!+\!K_2\theta^{\frac{\alpha}{1-\alpha}}\|w'-w\|^{\frac{\alpha}{1-\alpha}}\big)d\theta\nonumber\\
	&=\frac{1}{2}\|w'-w\|^2\big(K_0+K_1\|\nabla f(w)\|^{\alpha}\big)+K_2\|w'-w\|^{\frac{2-\alpha}{1-\alpha}}\int_0^1\theta^{\frac{1}{1-\alpha}}d\theta\nonumber\\
	&\le\frac{1}{2}\|w'-w\|^2\big(K_0+K_1\|\nabla f(w)\|^{\alpha}+2K_2\|w'-w\|^{\frac{\alpha}{1-\alpha}}\big),\nonumber
\end{align}
where (i) uses eq. \eqref{PS_poly_up} with $w'$ replaced by $w_{\theta}:=\theta w'+(1-\theta)w$.

\subsection{Proof of Item \ref{item:fup_exp}}
Since $f\in \mathcal{L}_{\text{sym}}^*(\alpha)$ for $\alpha\in (0,1)$, eq. \eqref{PS_exp_up} holds based on item \ref{item:df_bound_exp} of Proposition \ref{prop:PS_equiv}. Hence, for any $w,w'\in\mathbb{R}^d$, we prove eq. \eqref{PS_poly_fup} as follows.
\begin{align}
	&f(w')-f(w)-\nabla f(w)^{\top}(w'-w)\nonumber\\
	&=\int_0^1 \big(\nabla f(w_{\theta})-\nabla f(w)\big)^{\top}(w'-w) d\theta\nonumber\\
	&\le\int_0^1 \|\nabla f(w_{\theta})-\nabla f(w)\| \|w'-w\| d\theta\nonumber\\
	&\overset{(i)}{\le}\int_0^1 \|w_{\theta}-w\|\big(L_0+L_1\|\nabla f(w)\|\big)\exp\big(L_1\|w_{\theta}-w\|\big)\|w'-w\| d\theta\nonumber\\
	&\le\int_0^1 \theta\|w'-w\|^2\big(L_0+L_1\|\nabla f(w)\|\big)\exp\big(L_1\|w'-w\|\big)d\theta\nonumber\\
	&=\frac{1}{2}\|w'-w\|^2\big(L_0+L_1\|\nabla f(w)\|\big)\exp\big(L_1\|w'-w\|\big),\nonumber
\end{align}
where (i) uses eq. \eqref{PS_exp_up} with $w'$ replaced by $w_{\theta}:=\theta w'+(1-\theta)w$. 





\section{Proof of Proposition \ref{thm_phase} and Proposition \ref
	{prop:eg_avg}}\label{supp:DRO_phase}

\subsection{Proof for Phase Retrieval Problem}

The objective function \eqref{obj_phase} of phase retrieval problem can be rewritten in the stochastic form $f(z)=\mathbb{E}_{\xi} f_{\xi}(z)$ where $\xi$ is obtained from $\{1,2,\ldots,m\}$ uniformly at random and
\begin{align}
	f_{\xi}(z):=\frac{1}{2} (y_{\xi} - |a_{\xi}^{\top}z|^2)^2.\nonumber
\end{align}

To prove that $f\in\mathcal{L}_{\text{sym}}^*(\frac{2}{3})$ and $f\in\mathbb{E}\mathcal{L}_{\text{sym}}^*(\frac{2}{3})$ respectively required by Proposition \ref{thm_phase} and Proposition \ref
{prop:eg_avg}, it suffices to prove that $f_{\xi}\in\mathcal{L}_{\text{sym}}^*(\frac{2}{3})$ for every sample $\xi$.

For any $z\in\mathbb{R}^d$ and sample $\xi$, the gradient $\nabla f_{\xi}(z)=\frac{1}{2} (|a_{\xi}^{\top}z|^2-y_{\xi})(a_{\xi}a_{\xi}^{\top})z$ satisfies
\begin{align}
	\|\nabla f_{\xi}(z)\|^{\frac{2}{3}}&=\frac{1}{2^{\frac{2}{3}}}\big\|(|a_{\xi}^{\top}z|^2-y_{\xi})(a_{\xi}a_{\xi}^{\top})z\big\|^{\frac{2}{3}}\nonumber\\
	&\ge\frac{1}{2} \big||a_{\xi}^{\top}z|^3-y_{\xi}|a_{\xi}^{\top}z|\big|^{\frac{2}{3}}\|a_{\xi}\|^{\frac{2}{3}}\nonumber\\
	&\overset{(ii)}{\ge} \frac{1}{2}\big(|a_{\xi}^{\top}z|^2-|y_{\xi}||a_{\xi}^{\top}z|^{\frac{2}{3}}\big)\big|\|a_{\xi}\|^{\frac{2}{3}}\nonumber\\
	&\overset{(iii)}{\ge} \frac{1}{3}\big(|a_{\xi}^{\top}z|^2-|y_{\xi}|^{\frac{3}{2}}\big)\big|\|a_{\xi}\|^{\frac{2}{3}}\label{DRO_2b3}
\end{align}
where (i) applies Jensen's inequality, (ii) uses the inequality that $|a-b|^{\frac{2}{3}}\ge |a|^{\frac{2}{3}}-|b|^{\frac{2}{3}}$ for any $a,b\in\mathbb{R}$, (iii) uses $|y_{\xi}|a^{\frac{2}{3}}\le \frac{1}{3}a^2+\frac{2}{3}|y_{\xi}|^{\frac{3}{2}}$ for any $a\ge 0$ based on Young's inequality. 

For any $z,z'\in\mathbb{R}^d$, we obtain the following inequality which proves that $f_{\xi}\in\mathcal{L}_{\text{sym}}^*(\frac{2}{3})$ as desired. 
\begin{align}
	&\|\nabla f_{\xi}(z')-\nabla f_{\xi}(z)\|\nonumber\\
	&=\frac{1}{2}\big\|(|a_{\xi}^{\top}z'|^2-y_{\xi})(a_{\xi}a_{\xi}^{\top})z'-(|a_{\xi}^{\top}z|^2-y_{\xi})(a_{\xi}a_{\xi}^{\top})z\big\|\nonumber\\
	&\le \frac{1}{4}\big\|(|a_{\xi}^{\top}z'|^2+|a_{\xi}^{\top}z|^2-2y_{\xi})(a_{\xi}a_{\xi}^{\top})(z'-z)+(|a_{\xi}^{\top}z'|^2-|a_{\xi}^{\top}z|^2)(a_{\xi}a_{\xi}^{\top})(z'+z)\big\|\nonumber\\
	&\overset{(i)}{\le} \frac{1}{4}\|a_{\xi}\|^2\big(|a_{\xi}^{\top}z'|^2+|a_{\xi}^{\top}z|^2+2|y_{\xi}|\big)\|z'-z\| +\frac{1}{4}\|a_{\xi}\|^2(|a_{\xi}^{\top}z'|+|a_{\xi}^{\top}z|)^2\|z'-z\|\nonumber\\
	&\overset{(ii)}{\le} \frac{1}{4}\|z'-z\|\|a_{\xi}\|^2\big(3|a_{\xi}^{\top}z'|^2+3|a_{\xi}^{\top}z|^2+2|y_{\xi}|\big)\nonumber\\
	&\le \frac{1}{4}\|z'-z\|\|a_{\xi}\|^{\frac{4}{3}}\|a_{\xi}\|^{\frac{2}{3}}\big(3|a_{\xi}^{\top}z'|^2+3|a_{\xi}^{\top}z|^2-3|y_{\xi}|-3|y_{\xi}|+8|y_{\xi}|\big)\nonumber\\
	&\overset{(iii)}{\le} \|z'-z\|\Big(\frac{9}{4}a_{\max}^{\frac{4}{3}}\|\nabla f_{\xi}(z')\|^{\frac{2}{3}}+\frac{9}{4}a_{\max}^{\frac{4}{3}}\|\nabla f_{\xi}(z)\|^{\frac{2}{3}}+2y_{\max}a_{\max}^2\Big)\nonumber\\
	&\le \|z'-z\|\Big(\frac{9}{4}a_{\max}^{\frac{4}{3}}\max_{\theta\in[0,1]}\big\|\nabla f_{\xi}\big(\theta z'+(1-\theta)z\big)\big\|^{\frac{2}{3}}+2y_{\max}a_{\max}^2\Big)
\end{align}
where (i) uses trianagular inequality, $\|a_{\xi}a_{\xi}^{\top}\|=\|a_{\xi}\|^2$, $|y_{\xi}|\le 1$ and the following inequality, (ii) uses $(|a_{\xi}^{\top}z'|+|a_{\xi}^{\top}z|)^2\le 2|a_{\xi}^{\top}z'|^2+2|a_{\xi}^{\top}z|^2$, (iii) uses eq. \eqref{DRO_2b3} and denotes that $y_{\max}:=\max_{1\le r\le m}|y_r|$ and that $a_{\max}:=\max_{1\le r\le m}\|a_r\|$. 
\begin{align}
	\big||a_{\xi}^{\top}z'|^2\!-\!|a_{\xi}^{\top}z|^2\big|\!=\!(|a_{\xi}^{\top}z'|\!+\!|a_{\xi}^{\top}z|)(|a_{\xi}^{\top}z'|\!-\!|a_{\xi}^{\top}z|)\!\le\! (|a_{\xi}^{\top}z'|\!+\!|a_{\xi}^{\top}z|)\|a_{\xi}^{\top}(z'-z)\|\!\le\!\|a_{\xi}^{\top}\|(|a_{\xi}^{\top}z'|\!+\!|a_{\xi}^{\top}z|)\|z'\!-\!z\|.\nonumber
\end{align}



\subsection{Proof for DRO Problem}
We adopt the following assumptions from \cite{jin2021non}: 
\begin{itemize}
	\item $\ell_{\xi}$ is $G$-Lipschitz continuous and $L$-smooth. 
	\item $\mathbb{E}\big(\ell_{\xi}(x)-\ell(x)\big)^2\le\sigma^2$ where $\ell(x):=\mathbb{E}\ell_{\xi}(x)$
	\item $\psi$ is a non-negative convex function with $\psi(1)=0$ and $\psi(t)=+\infty$ for all $t<0$, and $\psi^*$ is $M$-smooth.
\end{itemize}

Then we rewrite the objective function \eqref{DRO2} as $L(x, \eta)=\mathbb{E}L_{\xi}(x, \eta)$ where
\begin{align}
	L_{\xi}(x, \eta):= \lambda \psi^*\left(\frac{\ell_{\xi}(x)-\eta}{\lambda}\right)+\eta. \label{DRO_xi}
\end{align}
The gradient $\nabla L_{\xi}=\big[\nabla_x  L_{\xi};\frac{\partial}{\partial \eta} L_{\xi}\big]$ can be computed as follows.

To prove that $L\in\mathbb{E}\mathcal{L}_{\text{sym}}^*(1)$ required by Proposition \ref
{prop:eg_avg}, it suffices to prove that $L_{\xi}\in\mathcal{L}_{\text{sym}}^*(1)$ for every sample $\xi$. 
\begin{align}
	&\nabla_x  L_{\xi}(x,\eta)={\psi^*}'\Big(\frac{\ell_{\xi}(x)-\eta}{\lambda}\Big)\nabla\ell_{\xi}(x),\label{Lhat_dx}\\
	&\frac{\partial}{\partial \eta} L_{\xi}(x,\eta)=1-{\psi^*}'\Big(\frac{\ell_{\xi}(x)-\eta}{\lambda}\Big).\label{Lhat_deta}
\end{align}
Hence, for any $(x', \eta'), (x, \eta)\in\mathbb{R}^d\times\mathbb{R}$, $\nabla  L_{\xi}(x', \eta')-\nabla  L_{\xi}(x, \eta)=A+B$ where
\begin{align}
	&A=\Big[{\psi^*}'\Big(\frac{\ell_{\xi}(x)-\eta}{\lambda}\Big)\big(\nabla\ell_{\xi}(x')-\nabla\ell_{\xi}(x)\big);0\Big] \nonumber\\
	&B=\Big[{\psi^*}'\Big(\frac{\ell_{\xi}(x')-\eta'}{\lambda}\Big)\!-\!{\psi^*}'\Big(\frac{\ell_{\xi}(x)-\eta}{\lambda}\Big)\Big]\big[\nabla\ell_{\xi}(x');-1\big].\nonumber
\end{align}
Therefore, we can prove that $L_{\xi}\in\mathcal{L}_{\text{sym}}^*(1)$ as follows.
\begin{align}
	&\big\|\nabla  L_{\xi}(x', \eta')-\nabla  L_{\xi}(x, \eta)\big\|\nonumber\\
	&\le \|A\|+\|B\|\nonumber\\
	&\le \Big|{\psi^*}'\Big(\frac{\ell_{\xi}(x)-\eta}{\lambda}\Big)\Big|\big\|\nabla\ell_{\xi}(x')-\nabla\ell_{\xi}(x)\big\|+\Big|{\psi^*}'\Big(\frac{\ell_{\xi}(x')-\eta'}{\lambda}\Big)\!-\!{\psi^*}'\Big(\frac{\ell_{\xi}(x)-\eta}{\lambda}\Big)\Big|\sqrt{\|\nabla\ell_{\xi}(x')\|^2+1}\nonumber\\
	&\overset{(i)}{\le} \Big|1-\frac{\partial}{\partial \eta} L_{\xi}(x,\eta)\Big|L\|x'-x\|+\frac{M}{\lambda}\big|\ell_{\xi}(x')-\eta'-\big(\ell_{\xi}(x)-\eta\big)\big|\sqrt{G^2+1} \nonumber\\
	&\overset{(ii)}{\le} \Big(L+L\Big|\frac{\partial}{\partial \eta} L_{\xi}(x,\eta)\Big|\Big)\|x'-x\|+\frac{M}{\lambda}\sqrt{G^2+1}\big[G\|x'-x\|+|\eta'-\eta|\big]\nonumber\\
	&\overset{(iii)}{\le} \Big(L+\frac{2M(G+1)^2}{\lambda}+L\|\nabla  L_{\xi}(x, \eta)\|\Big)\|(x'-x,\eta'-\eta)\|
\end{align}
where (i) uses eq. \eqref{Lhat_deta}, and the above assumptions that $\ell_{\xi}$ is $G$-Lipschitz, $L$-smooth and that $\psi^*$ is $M$-smooth, (ii) uses the above assumptions that $\ell_{\xi}$ is $G$-Lipschitz, and (iii) uses $\|x'-x\|+\|\eta'-\eta\|\le \sqrt{2}\|(x'-x,\eta'-\eta)\|$ and $\Big|\frac{\partial}{\partial \eta} L_{\xi}(x,\eta)\Big|\le\|\nabla  L_{\xi}(x, \eta)\|$.

\section{Proof of Theorem \ref{thm:GDconv}}\label{sec:GDProof}
We will first prove the following lemma which will be used in the proof of Theorem \ref{thm:GDconv}.
\begin{lemma}\label{lemma:young}
	For any $x\ge 0$, $C\in [0,1]$, $\Delta>0$ and $0\le \omega\le \omega'$ such that $\Delta\ge \omega'-\omega$, the following inequality holds
	\begin{align}\label{young}
		Cx^{\omega}\le x^{\omega'}+C^{\frac{\omega'}{\Delta}}
	\end{align}
\end{lemma}
\begin{proof}[Proof of Lemma \ref{lemma:young}]
	We consider three cases: $\omega=0$, $\omega'=\omega>0$ and $\omega'>\omega>0$.
	
	(Case I) When $\omega=0$, $\Delta\ge \omega'$ and $\Delta>0$ imply that $\frac{\omega'}{\Delta}\in [0,1]$, so $Cx^\omega=C\le C^{\frac{\omega'}{\Delta}}$, which implies eq. \eqref{young}. 
	
	(Case II) When $\omega'=\omega>0$, $Cx^{\omega}\le x^{\omega}=x^{\omega'}$, which implies eq. \eqref{young}. 
	
	(Case III) When $\omega'>\omega>0$, by applying Young's inequality with $p=\frac{\omega'}{\omega}>1$ and $q=\frac{\omega'}{\omega'-\omega}>1$ which satisfy $\frac{1}{p}+\frac{1}{q}=1$, we prove eq. \eqref{young} as follows.
	\begin{align}
		Cx^{\omega}\le \frac{x^{p\omega}}{p}+\frac{C^q}{q}\le x^{\omega'}+C^{\frac{\omega'}{\omega'-\omega}}\le x^{\omega'}+C^{\frac{\omega'}{\Delta}}.\nonumber
	\end{align}
\end{proof}

Now we will prove Theorem \ref{thm:GDconv}. We omit the well-known case of $\beta=\alpha=0$ where GD is applied to $L$-smooth function $f\in\mathcal{L}$. Hence, we focus on the case of $\beta>0$. We first bound $f(w_{t+1})-f(w_t)$ in two cases: $\alpha\in(0,1)$ and $\alpha=1$.

(Case I) When $\alpha\in(0,1)$, eq. \eqref{PS_poly_fup} holds for $f\in\mathcal{L}_{\text{sym}}^*(\alpha)$. Hence, we have
\begin{align}
	&f(w_{t+1})-f(w_t)\nonumber\\
	&\le\nabla f(w_t)^{\top}(w_{t+1}-w_t)+\frac{1}{2}\big(K_0+K_1\|\nabla f(w_t)\|^{\alpha}\big)\|w_{t+1}-w_t\|^2+K_2\|w_{t+1}-w_t\|^{\frac{2-\alpha}{1-\alpha}}\nonumber\\ 
	&\overset{(i)}{=}-\gamma\|\nabla f(w_t)\|^{2-\beta}+\frac{\gamma}{6}\big(3K_0\gamma\cdot\|\nabla f(w_t)\|^{2-2\beta}+3K_1\gamma\cdot\|\nabla f(w_t)\|^{2+\alpha-2\beta}+6K_2\gamma^{\frac{1}{1-\alpha}}\cdot\|\nabla f(w_t)\|^{\frac{(2-\alpha)(1-\beta)}{1-\alpha}}\big)\nonumber\\
	&\overset{(ii)}{\le}-\gamma\|\nabla f(w_t)\|^{2-\beta}+\frac{\gamma}{6}\big(3\|\nabla f(w_t)\|^{2-\beta}+(3K_0\gamma)^{\frac{2}{\beta}-1}+(3K_1\gamma)^{\frac{2}{\beta}-1}+(6K_2\gamma)^{\frac{2}{\beta}-1}\big)\nonumber\\
	&\overset{(iii)}{\le}-\frac{\gamma}{2}\|\nabla f(w_t)\|^{2-\beta}+\gamma^{\frac{2}{\beta}}(3K_0+3K_1+6K_2)^{\frac{2}{\beta}-1}\nonumber\\
	&\overset{(iv)}{\le}-\frac{\gamma}{2}\|\nabla f(w_t)\|^{2-\beta}+\frac{\gamma}{4}\epsilon^{2-\beta}\nonumber\label{fdec_poly}
\end{align}
where (i) uses the update rule $w_{t+1}=w_t-\gamma\frac{\nabla f(w_t)}{\|\nabla f(w_t)\|^{\beta}}$ of Algorithm \ref{algo: GD} ($\beta$-GD), (ii) uses $\gamma^{\frac{1}{1-\alpha}}\le 1$ and applies Lemma \ref{lemma:young} three times respectively with $x=\|\nabla f(w_t)\|$, $C=3K_0\gamma,3K_1\gamma,6K_2\gamma$ ($C\in[0,1]$ since $\gamma=\frac{\epsilon^{\beta}}{12(K_0+K_1+2K_2)+1}$ and $\epsilon\in (0,1)$), $\Delta=\beta$, $\omega=2-2\beta,2+\alpha-2\beta,\frac{(2-\alpha)(1-\beta)}{1-\alpha}$, $\omega'=2-\beta$, (iii) uses the inequality that $a^{\tau}+b^{\tau}+c^{\tau}\le (a+b+c)^{\tau}$ for $\tau=\frac{2}{\beta}-1>1$ and any $a,b,c\ge 0$, and (iv) uses $\gamma=\frac{\epsilon^{\beta}}{12(K_0+K_1+2K_2)+1}$. 

(Case II) When $\alpha=1$, we have $\beta=1$ and eq. \eqref{PS_exp_fup} holds for $f\in\mathcal{L}_{\text{sym}}^*(1)$. Hence, we have 
\begin{align}
	&f(w_{t+1})-f(w_t)\nonumber\\
	&\le\nabla f(w_t)^{\top}(w_{t+1}-w_t)+\frac{1}{2}\|w_{t+1}-w_t\|^2\big(L_0+L_1\|\nabla f(w_t)\|\big)\exp\big(L_1\|w_{t+1}-w_t\|\big)\nonumber\\ 
	&\overset{(i)}{=}-\gamma\|\nabla f(w_t)\|+\frac{\gamma^2}{2}\big(L_0+L_1\|\nabla f(w_t)\|\big)\exp(L_1\gamma)\nonumber\\
	&\overset{(ii)}{\le}-\frac{\gamma}{2}\|\nabla f(w_t)\|+L_0\gamma^2\nonumber\\
	&\overset{(iii)}{\le}-\frac{\gamma}{2}\|\nabla f(w_t)\|^{2-\beta}+\frac{\gamma}{4}\epsilon
\end{align}
where (i) uses the update rule $w_{t+1}=w_t-\gamma\frac{\nabla f(w_t)}{\|\nabla f(w_t)\|}$ of Algorithm \ref{algo: GD} ($\beta$-GD with $\beta=1$) and (ii) and (iii) use $\gamma=\frac{\epsilon}{4L_0+1}\le \frac{1}{2L_1}$. Note that eq. \eqref{fdec_poly} holds in both cases. Therefore, by telescoping eq. \eqref{fdec_poly} and rearranging it, we obtain that
\begin{align}
	\mathbb{E}_{\widetilde{T}}\|\nabla f(w_{\widetilde{T}})\|^{2-\beta}&=\frac{1}{T}\sum_{t=1}^T\|\nabla f(w_t)\|^{2-\beta}\nonumber\\
	&\le \frac{2}{T\gamma}\big(f(w_0)-f^*\big)+\frac{1}{2}\epsilon^{2-\beta}, \nonumber
\end{align}
where (i) uses $\gamma=\frac{\epsilon^{\beta}}{12(K_0+K_1+2K_2)+1}$ and $f(w_T)\ge f^*:=\min_{w\in\mathbb{R}^d}f(w)$. By applying Lyapunov inequality, the above inequality implies convergence rate \eqref{GD_rate} as follows.
\begin{align}
	\mathbb{E}_{\widetilde{T}}\|\nabla f(w_{\widetilde{T}})\|&\le\big(\mathbb{E}_{\widetilde{T}}\|\nabla f(w_{\widetilde{T}})\|^{2-\beta}\big)^{\frac{1}{2-\beta}}\nonumber\\
	&\le\Big(\frac{2}{T\gamma}\big(f(w_0)-f^*\big)+\frac{1}{2}\epsilon^{2-\beta}\Big)^{\frac{1}{2-\beta}}, \nonumber\\
	&\overset{(i)}{\le}\Big(\frac{2}{T\gamma}\Big)^{\frac{1}{2-\beta}}\big(f(w_0)-f^*\big)^{\frac{1}{2-\beta}}+\Big(\frac{1}{2}\Big)^{\frac{1}{2-\beta}}\epsilon \nonumber\\
	&\le\Big(\frac{2}{T\gamma}\Big)^{\frac{1}{2-\beta}}\big(f(w_0)-f^*\big)^{\frac{1}{2-\beta}}+\frac{1}{2}\epsilon \nonumber
\end{align}
where (i) uses $(a+b)^{\tau}\le a^{\tau}+b^{\tau}$ for $\tau=\frac{1}{2-\beta}\in[0,1]$ and any $a,b\ge 0$. 

Then, substituting $T=\frac{4}{\gamma}$ into the above convergence rate, we obtain that $\mathbb{E}_{\widetilde{T}}\|\nabla f(w_{\widetilde{T}})\|\le\epsilon$.

\section{Proof of Theorem \ref{thm: GD-diverge}}
We consider the following two cases. 

(Case I) When $\alpha\in(0,1)$, consider the convex function $f(w):=|w|^{\frac{2-\alpha}{1-\alpha}}$ with unique minimizer $w=0$ and derivative $f'(w)=\frac{2-\alpha}{1-\alpha}|w|^{\frac{1}{1-\alpha}}\text{sgn}(w)$. Based on item \ref{item:sym_poly} of Proposition \ref{thm: func_class}, $f\in\mathcal{L}_{\text{sym}}^*(\alpha)$. Applying $\beta$-GD to this function yields that
\begin{align}
	w_{t+1}&=w_t-\frac{\gamma f'(w_t)}{|f'(w_t)|^{\beta}}=w_t-\gamma\Big(\frac{2-\alpha}{1-\alpha}\Big)^{1-\beta} |w_t|^{\frac{1-\beta}{1-\alpha}}\text{sgn}(w_t).\nonumber
\end{align}
Note that $0\le\beta<\alpha<1$. Hence, if $|w_t|>C$ with constant $C:=\Big(\frac{3(1-\alpha)}{\gamma(2-\alpha)}\Big)^{\frac{1-\alpha}{\alpha-\beta}}>0$, we have $\gamma\frac{2-\alpha}{1-\alpha}|w_t|^{\frac{1-\beta}{1-\alpha}}>3|w_t|$ and thus $|w_{t+1}|>2|w_t|$. Therefore, if $|w_0|>C$, by induction we obtain that $|w_t|>2^tC$ for any $t$, and thus $|f'(w_t)|>2^{\frac{t}{1-\alpha}}C^{\frac{1}{1-\alpha}}$, $f(w_t)>2^{\frac{t(2-\alpha)}{1-\alpha}}C^{\frac{2-\alpha}{1-\alpha}}$, which means $\beta$-GD diverges. 

(Case II) When $\alpha=1$, consider the convex function $f(w):=e^w+e^{-w}$ with unique minimizer $w=0$ and derivative $f'(w):=e^w-e^{-w}=(e^{|w|}-e^{-|w|})\text{sgn}(w)$. Based on item \ref{item:sym_exp} of Proposition \ref{thm: func_class}, $f\in\mathcal{L}_{\text{sym}}^*(1)$. Applying $\beta$-GD to this function yields that
\begin{align}
	w_{t+1}&=w_t-\frac{\eta f'(w_t)}{|f'(w_t)|^{\beta}}=w_t-\eta \big(e^{|w_t|}-e^{-|w_t|}\big)^{1-\beta}\text{sgn}(w_t).
\end{align}

Since $\beta<\alpha=1$, $|w_t|^{-1}\big(e^{|w_t|}-e^{-|w_t|}\big)^{1-\beta}\to+\infty$ as $|w_t|\to +\infty$. Hence, there exists a constant $C>1$ such that $\big(e^{|w_t|}-e^{-|w_t|}\big)^{1-\beta}>3|w_t|$ for $|w_t|>C$. Therefore, $|w_{t+1}|>2|w_t|$. Therefore, if $|w_0|>C$, by induction we obtain that $|w_t|>2^tC$ for any $t$, and thus $f(w_t)>|f'(w_t)|=e^{|w_t|}-e^{-|w_t|}\ge\frac{1}{2}e^{|w_t|}>\frac{1}{2}\exp(2^t C)$, which means $\beta$-GD diverges.

\section{Proof of Proposition \ref{prop:avgsmooth}}
\subsection{Proof of Item \ref{item:avg_poly}}
First, we will prove eq. \eqref{PSavg_poly_up} given  $f\in\mathbb{E}\mathcal{L}_{\text{sym}}^*(\alpha)$. Note that eq. \eqref{theta_int_avg} holds for $f\in\mathbb{E}\mathcal{L}_{\text{sym}}^*(\alpha)$, i.e., 
\begin{align}
	\mathbb{E}_{\xi}\|\nabla f_{\xi}(w_{\theta})-\nabla f_{\xi}(w)\|^2&\le \theta^2\|w'-w\|^2\mathbb{E}_{\xi}\int_0^1 \big(L_0+L_1\|\nabla f_{\xi}(w_{\theta u})\|^{\alpha}\big)^2du\nonumber\\
	&\overset{(i)}{=} \theta\|w'-w\|^2\mathbb{E}_{\xi}\int_0^{\theta} \big(L_0+L_1\|\nabla f_{\xi}(w_{u'})\|^{\alpha}\big)^2du'\nonumber\\
	&\overset{(ii)}{\le} G(\theta)\|w'-w\|^2\label{theta_int_avg2}
\end{align}
where (i) uses change of variables $u'
=\theta u$ and (ii) denotes $G(\theta):=\mathbb{E}_{\xi}\int_0^{\theta} \big(L_0+L_1\|\nabla f_{\xi}(w_{u'})\|^{\alpha}\big)^2du'$ and uses $\theta\le 1$. Then 
\begin{align}
	G'(\theta)&=\mathbb{E}_{\xi} \big(L_0+L_1\|\nabla f_{\xi}(w_{\theta})\|^{\alpha}\big)^2\nonumber\\
	&\overset{(i)}{\le} 2L_0^2+2L_1^2\mathbb{E}_{\xi}\|\nabla f_{\xi}(w_{\theta})\|^{2\alpha}\nonumber\\
	&\overset{(ii)}{\le} 2L_0^2+4L_1^2 \mathbb{E}_{\xi}\|\nabla f_{\xi}(w)\|^{2\alpha}+4L_1^2 \mathbb{E}_{\xi}\|\nabla f_{\xi}(w_{\theta})-\nabla f_{\xi}(w)\|^{2\alpha}\nonumber\\
	&\overset{(iii)}{\le} 2L_0^2+4L_1^2 \mathbb{E}_{\xi}\|\nabla f_{\xi}(w)\|^{2\alpha}+4L_1^2 \big(\mathbb{E}_{\xi}\|\nabla f_{\xi}(w_{\theta})-\nabla f_{\xi}(w)\|^2\big)^{\alpha}\nonumber\\
	&\overset{(iv)}{\le} 2L_0^2+4L_1^2 \mathbb{E}_{\xi}\|\nabla f_{\xi}(w)\|^{2\alpha}+4L_1^2G(\theta)^{\alpha}\|w'-w\|^{2\alpha}\nonumber\\
	&\overset{(v)}{\le} 3\big(A+BG(\theta)\big)^{\alpha}, \label{G_ODE2}
\end{align}
where (i) use the inequality that $(a+b)^2\le 2a^2+2b^2$ for any $a,b\ge 0$, (ii) uses the inequality that $\|v'+v\|^{2\alpha}\le 2\|v'\|^{2\alpha}+2\|v\|^{2\alpha}$ for any $v,v'\in\mathbb{R}^d$ and $\alpha\in[0,1]$, (iii) uses Jensen's inequality that $\mathbb{E}(X^{\alpha})\le(\mathbb{E}X)^{\alpha}$ where $X=\|\nabla f_{\xi}(w_{\theta})-\nabla f_{\xi}(w)\|^2$, (iv) uses eq. \eqref{theta_int_avg2}, and (v) uses Jensen's inequality that $a^{\alpha}+b^{\alpha}+c^{\alpha}\le 3(a+b+c)^{\alpha}$ for any $a,b,c\ge 0$ and denotes that $A:=(2L_0^2)^{\frac{1}{\alpha}}+(4L_1^2\mathbb{E}_{\xi}\|\nabla f_{\xi}(w)\|^{2\alpha})^{\frac{1}{\alpha}}$, $B:=(4L_1^2)^{\frac{1}{\alpha}}\|w'-w\|^2$. When $\alpha\in(0,1)$, rearranging the above inequality yields that
\begin{align}
	3B(1-\alpha)\ge \frac{B(1-\alpha)G'(\theta)}{\big(A+BG(\theta)\big)^{\alpha}}=\frac{d}{d\theta}\big(A+BG(\theta)\big)^{1-\alpha}\nonumber
\end{align}
Integrating the above inequality over $\theta\in[0,1]$ yields that
\begin{align}
	\big(A+BG(1)\big)^{1-\alpha}&\le 3B(1-\alpha)+\big(A+BG(0)\big)^{1-\alpha}\le 3B+A^{1-\alpha}\le 2\big((3B)^{\frac{1}{1-\alpha}}+A\big)^{1-\alpha}.\nonumber
\end{align}
where (i) applies Jensen's inequality to the concave function $x^{1-\alpha}$. Rearranging the above inequality yields that
\begin{align} 
	BG(1)&\le 2^{\frac{1}{1-\alpha}}\big((3B)^{\frac{1}{1-\alpha}}+A\big)-A\nonumber\\
	&\le 6^{\frac{1}{1-\alpha}}B^{\frac{1}{1-\alpha}}+A(2^{\frac{1}{1-\alpha}}) \nonumber\\
	&\le 6^{\frac{1}{1-\alpha}}(4L_1^2)^{\frac{1}{\alpha(1-\alpha)}}\|w'-w\|^{\frac{2}{1-\alpha}}+2^{\frac{1}{1-\alpha}}(2L_0^2)^{\frac{1}{\alpha}}+2^{\frac{1}{1-\alpha}}(4L_1^2\mathbb{E}_{\xi}\|\nabla f_{\xi}(w)\|^{2\alpha})^{\frac{1}{\alpha}}.\nonumber
\end{align}
Substituting the above inequality into eq. \eqref{theta_int_avg2}, we obtain that
\begin{align}
	\mathbb{E}_{\xi}\|\nabla f_{\xi}(w_{\theta})-\nabla f_{\xi}(w)\|^2&\le  G(\theta)\|w'-w\|^2\nonumber\\
	&\le (4L_1^2)^{-\frac{1}{\alpha}}BG(1)\nonumber\\
	&\le (24L_1^2)^{\frac{1}{1-\alpha}}\|w'-w\|^{\frac{2}{1-\alpha}}+2^{\frac{1}{1-\alpha}}\Big(\frac{L_0^2}{2L_1^2}\Big)^{\frac{1}{\alpha}}+2^{\frac{1}{1-\alpha}}(\mathbb{E}_{\xi}\|\nabla f_{\xi}(w)\|^{2\alpha})^{\frac{1}{\alpha}} \label{Ediff}
\end{align}
Therefore, 
\begin{align}
	\mathbb{E}\|\nabla f_{\xi}(w_{\theta})\|^{2\alpha}&\overset{(i)}{\le} 2\mathbb{E}\|\nabla f_{\xi}(w_{\theta})-\nabla f_{\xi}(w)\|^{2\alpha}+2\mathbb{E}\|\nabla f_{\xi}(w)\|^{2\alpha}\nonumber\\
	&\overset{(ii)}{\le}2\big(\mathbb{E}\|\nabla f_{\xi}(w_{\theta})-\nabla f_{\xi}(w)\|^2\big)^{\alpha}+2\mathbb{E}\|\nabla f_{\xi}(w)\|^{2\alpha}\nonumber\\
	&\overset{(iii)}{\le} 2(24L_1^2)^{\frac{\alpha}{1-\alpha}}\|w'-w\|^{\frac{2\alpha}{1-\alpha}}+2^{\frac{\alpha}{1-\alpha}}L_0^2L_1^{-2}+\big(2^{\frac{1}{1-\alpha}}+2\big)\mathbb{E}_{\xi}\|\nabla f_{\xi}(w)\|^{2\alpha}\label{Edf2alpha}
\end{align}
where (i) uses the inequality that $\|v'+v\|^{2\alpha}\le 2\|v'\|^{2\alpha}+2\|v\|^{2\alpha}$ for any $v,v'\in\mathbb{R}^d$ and $\alpha\in(0,1)$, (ii) uses Jensen's inequality that $\mathbb{E}(X^{\alpha})\le(\mathbb{E}X)^{\alpha}$ where $X=\|\nabla f_{\xi}(w_{\theta})-\nabla f_{\xi}(w)\|^2$, and (iii) uses eq. \eqref{Ediff} and the inequality that $(a+b+c)^{\alpha}\le a^{\alpha}+b^{\alpha}+c^{\alpha}$ for any $a,b,c\ge 0$ and $\alpha\in[0,1]$. Note that \eqref{avg_int} holds for $f\in\mathbb{E}\mathcal{L}_{\text{sym}}^*$, so we have
\begin{align}
	&\mathbb{E}_{\xi}\|\nabla f_{\xi}(w')-\nabla f_{\xi}(w)\|^2\nonumber\\
	&\le\|w'-w\|^2\mathbb{E}_{\xi}\int_0^1\big(L_0+L_1\|\nabla f_{\xi}(w_{\theta})\|^{\alpha}\big)^2d\theta\nonumber\\
	&\le 2\|w'-w\|^2\int_0^1 L_0^2+L_1^2\mathbb{E}_{\xi}\|\nabla f_{\xi}(w_{\theta})\|^{2\alpha}d\theta\nonumber\\
	&\overset{(i)}{\le}\|w'-w\|^2\mathbb{E}_{\xi}\big(\overline{K}_0^2+\overline{K}_1^2\|\nabla f_{\xi}(w)\|^{2\alpha}+\overline{K}_2^2\|w'-w\|^{\frac{2\alpha}{1-\alpha}}\big)\nonumber\\
	&\overset{(ii)}{\le}\|w'-w\|^2\mathbb{E}_{\xi}\big(\overline{K}_0+\overline{K}_1\|\nabla f_{\xi}(w)\|^{\alpha}+\overline{K}_2\|w'-w\|^{\frac{\alpha}{1-\alpha}}\big)^2\nonumber
\end{align}
where (i) uses eq. \eqref{Edf2alpha} and denotes that $\overline{K}_0^2:=2^{\frac{4-2\alpha}{1-\alpha}}L_0^2\ge 2L_0^2(2^{\frac{1}{1-\alpha}}+1)$, $\overline{K}_1^2:=2^{\frac{4-2\alpha}{1-\alpha}}L_1^2\ge 2L_1^2(2^{\frac{1}{1-\alpha}}+2)$, $\overline{K}_2^2:=(25L_1^2)^{\frac{1}{1-\alpha}}\ge 4L_1^2(24L_1^2)^{\frac{\alpha}{1-\alpha}}$, and (ii) uses the inequality that $a^2+b^2+c^2\le(a+b+c)^2$ for any $a,b,c\ge 0$. This proves eq. \eqref{PSavg_poly_up}. 

Then, it remains to prove that $f\in\mathbb{E}\mathcal{L}_{\text{sym}}^*(\alpha)$ given eq. \eqref{PSavg_poly_up}. Then for any $w,w'\in\mathbb{R}^d$ and $n\in\mathbb{N}^+$, we have
\begin{align}
	&\mathbb{E}_{\xi}\|\nabla f_{\xi}(w')-\nabla f_{\xi}(w)\|^2\nonumber\\
	&=\mathbb{E}_{\xi}\Big\|\sum_{k=0}^{n-1} \big(\nabla f_{\xi}(w_{(k+1)/n})-\nabla f_{\xi}(w_{k/n})\big)\Big\|^2\nonumber\\
	&\overset{(i)}{\le} n\sum_{k=0}^{n-1}\mathbb{E}_{\xi}\big\|\nabla f_{\xi}(w_{(k+1)/n})-\nabla f_{\xi}(w_{k/n})\big\|^2\nonumber\\
	&\overset{(ii)}{\le}n \sum_{k=0}^{n-1} \|w_{(k+1)/n}-w_{k/n}\|^2\mathbb{E}_{\xi}\big(\overline{K}_0+\overline{K}_1\|\nabla f_{\xi}(w_{k/n})\|^{\alpha}+\overline{K}_2\|w_{(k+1)/n}-w_{k/n}\|^{\frac{\alpha}{1-\alpha}}\big)^2\nonumber\\
	&=\|w'-w\|^2\mathbb{E}_{\xi}\sum_{k=0}^{n-1} \frac{1}{n}\big(\overline{K}_0+\overline{K}_1\|\nabla f_{\xi}(w_{k/n})\|^{\alpha}+\overline{K}_2n^{-\frac{\alpha}{1-\alpha}}\|w'-w\|^{\frac{\alpha}{1-\alpha}}\big)^2,\nonumber
\end{align}
where (i) applies Jensen' inequality to the convex function $\|\cdot\|^2$ and (ii) uses eq. \eqref{PSavg_poly_up}. For any $\epsilon>0$, there exists $n_0>0$ such that $\overline{K}_2n^{-\frac{\alpha}{1-\alpha}}\|w'-w\|^{\frac{\alpha}{1-\alpha}}<\epsilon$ for any $n\ge n_0$. Therefore, taking limit superior of both sides of
the above inequality, we obtain that
\begin{align}
	&\mathbb{E}_{\xi}\|\nabla f_{\xi}(w')-\nabla f_{\xi}(w)\|^2\nonumber\\
	&\le\|w'-w\|^2\mathop{\lim\sup}\limits_{n\to+\infty}\mathbb{E}_{\xi}\sum_{k=0}^{n-1} \frac{1}{n}\big(\overline{K}_0+\overline{K}_1\|\nabla f_{\xi}(w_{k/n})\|^{\alpha}+\overline{K}_2n^{-\frac{\alpha}{1-\alpha}}\|w'-w\|^{\frac{\alpha}{1-\alpha}}\big)^2\nonumber\\
	&\overset{(i)}{\le} \|w'-w\|^2 \mathbb{E}_{\xi}\mathop{\lim\sup}\limits_{n\to+\infty}\sum_{k=0}^{n-1} \frac{1}{n}\big(\overline{K}_0+\epsilon+\overline{K}_1\|\nabla f_{\xi}(w_{k/n})\|^{\alpha}\big)^2\nonumber\\
	&=\|w'-w\|^2 \mathbb{E}_{\xi}\int_0^1 \big(\overline{K}_0+\epsilon+\overline{K}_1\|\nabla f_{\xi}(w_{\theta})\|^{\alpha}\big)^2d\theta\nonumber
\end{align}
where (i) uses Fatou's lemma. Letting $\epsilon\to +0$ in the above inequality, we obtain the following inequality, which proves that $f\in\mathbb{E}\mathcal{L}_{\text{sym}}^*(\alpha)$ based on Lemma \ref{lemma:avg_int}
\begin{align}
	&\mathbb{E}_{\xi}\|\nabla f_{\xi}(w')-\nabla f_{\xi}(w)\|^2\le\|w'-w\|^2 \mathbb{E}_{\xi}\int_0^1 \big(\overline{K}_0+\overline{K}_1\|\nabla f_{\xi}(w_{\theta})\|^{\alpha}\big)^2d\theta\label{back2avg_poly}
\end{align}

\subsection{Proof of Item \ref{item:avg_exp}}
First, we will prove eq. \eqref{PSavg_exp_up} given  $f\in\mathbb{E}\mathcal{L}_{\text{sym}}^*(1)$. 
Note that eq. \eqref{G_ODE2} holds for any $f\in\mathbb{E}\mathcal{L}_{\text{sym}}^*(\alpha)$ with $\alpha\in[0,1]$. Substituting $\alpha=1$ into eq. \eqref{G_ODE2}, i.e.,
\begin{align}
	G'(\theta)\le 3A+3BG(\theta) \nonumber
\end{align}
where $G(\theta):=\mathbb{E}_{\xi}\int_0^{\theta} \big(L_0+L_1\|\nabla f_{\xi}(w_{u})\|\big)^2du$, $w_u:=uw'+(1-u)w$, $A:=2L_0^2+4L_1^2\mathbb{E}_{\xi}\|\nabla f_{\xi}(w)\|^2$ and $B:=4L_1^2\|w'-w\|^2$. Rearranging the above inequality yields that
\begin{align}
	3B\ge \frac{BG'(\theta)}{A+BG(\theta)}=\frac{d}{d\theta}\ln\big(A+BG(\theta)\big).\nonumber
\end{align}
Integrating the above inequality over $\theta\in[0,1]$, we obtain that
\begin{align}
	\ln\big(A+BG(1)\big)\le \ln\big(A+BG(0)\big)+3B=\ln A+3B.\nonumber
\end{align}
Hence, we have $BG(\theta)\le BG(1)\le A(e^{3B}-1)$. Substituting this inequality into eq. \eqref{theta_int_avg2}, we obtain that
\begin{align}
	\mathbb{E}_{\xi}\|\nabla f_{\xi}(w_{\theta})-\nabla f_{\xi}(w)\|^2 &\le G(\theta)\|w'-w\|^2 \nonumber\\
	&\le \frac{A}{4L_1^2}(e^{3B}-1)\nonumber\\
	&\le \Big(\frac{L_0^2}{2L_1^2}+\mathbb{E}_{\xi}\|\nabla f_{\xi}(w)\|^2\Big)\big(\exp(12L_1^2\|w'-w\|^2)-1\big). \label{Ediff2}
\end{align}
Therefore, 
\begin{align}
	\mathbb{E}\|\nabla f_{\xi}(w_{\theta})\|^2&\overset{(i)}{\le} 2\mathbb{E}\|\nabla f_{\xi}(w_{\theta})-\nabla f_{\xi}(w)\|^2+2\mathbb{E}\|\nabla f_{\xi}(w)\|^2\nonumber\\
	&\overset{(ii)}{\le} \Big(\frac{L_0^2}{L_1^2}+2\mathbb{E}_{\xi}\|\nabla f_{\xi}(w)\|^2\Big)\big(\exp(12L_1^2\|w'-w\|^2)-1\big)+2\mathbb{E}\|\nabla f_{\xi}(w)\|^2 \label{Edf2alpha2}
\end{align}
where (i) uses the inequality that $\|v'+v\|^2\le 2\|v'\|^2+2\|v\|^2$ for any $v,v'\in\mathbb{R}^d$ and (ii) uses eq. \eqref{Ediff2}. Note that \eqref{avg_int} holds for $f\in\mathbb{E}\mathcal{L}_{\text{sym}}^*(1)$. Hence, we prove eq. \eqref{PSavg_exp_up} as follows. 
\begin{align}
	&\mathbb{E}_{\xi}\|\nabla f_{\xi}(w')-\nabla f_{\xi}(w)\|^2\nonumber\\
	&\le\|w'-w\|^2\mathbb{E}_{\xi}\int_0^1\big(L_0+L_1\|\nabla f_{\xi}(w_{\theta})\|\big)^2d\theta\nonumber\\
	&\le 2\|w'-w\|^2\int_0^1 L_0^2+L_1^2\mathbb{E}_{\xi}\|\nabla f_{\xi}(w_{\theta})\|^2d\theta\nonumber\\
	&\overset{(i)}{\le}2\|w'-w\|^2\Big(L_0^2+(L_0^2+2L_1^2\mathbb{E}_{\xi}\|\nabla f_{\xi}(w)\|^2)\big(\exp(12L_1^2\|w'-w\|^2)-1\big)+2L_1^2\mathbb{E}\|\nabla f_{\xi}(w)\|^2\Big)\nonumber\\
	&=2\|w'-w\|^2(L_0^2+2L_1^2\mathbb{E}_{\xi}\|\nabla f_{\xi}(w)\|^2)\exp(12L_1^2\|w'-w\|^2),\nonumber
\end{align}
where (i) uses eq. \eqref{Edf2alpha2}. This proves eq. \eqref{PSavg_exp_up}. 

Finally, it remains to prove that $f\in\mathbb{E}\mathcal{L}_{\text{sym}}^*(\alpha)$ given eq. \eqref{PSavg_exp_up}. Then for any $w,w'\in\mathbb{R}^d$ and $n\in\mathbb{N}^+$, we have
\begin{align}
	&\mathbb{E}_{\xi}\|\nabla f_{\xi}(w')-\nabla f_{\xi}(w)\|^2\nonumber\\
	&=\mathbb{E}_{\xi}\Big\|\sum_{k=0}^{n-1} \big(\nabla f_{\xi}(w_{(k+1)/n})-\nabla f_{\xi}(w_{k/n})\big)\Big\|^2\nonumber\\
	&\overset{(i)}{\le} n\sum_{k=0}^{n-1}\mathbb{E}_{\xi}\big\|\nabla f_{\xi}(w_{(k+1)/n})-\nabla f_{\xi}(w_{k/n})\big\|^2\nonumber\\
	&\overset{(ii)}{\le}2n \sum_{k=0}^{n-1} \|w_{(k+1)/n}-w_{k/n}\|^2(L_0^2+2L_1^2\mathbb{E}_{\xi}\|\nabla f_{\xi}(w_{k/n})\|^2)\exp(12L_1^2\|w_{(k+1)/n}-w_{k/n}\|^2)\nonumber\\
	&=\|w'-w\|^2\sum_{k=0}^{n-1} \frac{1}{n}(L_0^2+2L_1^2\mathbb{E}_{\xi}\|\nabla f_{\xi}(w_{k/n})\|^2)\exp(12n^{-2}L_1^2\|w'-w\|^2),\nonumber
\end{align}
where (i) applies Jensen' inequality to the convex function $\|\cdot\|^2$ and (ii) uses eq. \eqref{PSavg_exp_up}. For any $\epsilon>0$, there exists $n_0>0$ such that $\exp(12n^{-2}L_1^2\|w'-w\|^2)<1+\epsilon$ for any $n\ge n_0$. Therefore, letting $n\to+\infty$ in the above inequality, we obtain that
\begin{align}
	&\mathbb{E}_{\xi}\|\nabla f_{\xi}(w')-\nabla f_{\xi}(w)\|^2\nonumber\\
	&\le(1+\epsilon)\|w'-w\|^2 \mathop{\lim\sup}\limits_{n\to+\infty}\sum_{k=0}^{n-1} \frac{1}{n}(L_0^2+2L_1^2\mathbb{E}_{\xi}\|\nabla f_{\xi}(w_{k/n})\|^2)\nonumber\\
	&=(1+\epsilon)\|w'-w\|^2 \int_0^1 (L_0^2+2L_1^2\mathbb{E}_{\xi}\|\nabla f_{\xi}(w_{\theta})\|^2) d\theta,\nonumber
\end{align}
where (i) uses Fatou's lemma. Letting $\epsilon\to +0$ in the above inequality, we obtain the following inequality, which proves that $f\in\mathbb{E}\mathcal{L}_{\text{sym}}^*(1)$ based on Lemma \ref{lemma:avg_int}
\begin{align}
	&\mathbb{E}_{\xi}\|\nabla f_{\xi}(w')-\nabla f_{\xi}(w)\|^2\le\|w'-w\|^2 \mathbb{E}_{\xi}\int_0^1 (L_0^2+2L_1^2\mathbb{E}_{\xi}\|\nabla f_{\xi}(w_{\theta})\|^2) d\theta.\label{back2avg_exp}
\end{align}

\subsection{Proof of Item \ref{item:avg2f}}
For any $f\in\mathbb{E}\mathcal{L}_{\text{sym}}^*(\alpha)$, we will prove that $f\in\mathcal{L}_{\text{sym}}^*(\alpha)$ in two cases: $\alpha\in(0,1)$ and $\alpha=1$. 

(Case I) When $\alpha\in(0,1)$, eq. \eqref{PSavg_poly_up} holds, so we have
\begin{align}
	&\|\nabla f(w')-\nabla f(w)\|=\big\|\mathbb{E}_{\xi}\big(\nabla f_{\xi}(w')-\nabla f_{\xi}(w)\big)\big\| \nonumber\\
	&\le\sqrt{\mathbb{E}_{\xi}\|\nabla f_{\xi}(w')-\nabla f_{\xi}(w)\|^2} \nonumber\\
	&\le\|w'-w\|\big(\overline{K}_0+\overline{K}_1\mathbb{E}_{\xi}\|\nabla f_{\xi}(w)\|^{\alpha}+\overline{K}_2\|w'-w\|^{\frac{\alpha}{1-\alpha}}\big)\nonumber\\
	&\overset{(i)}{\le}\|w'-w\|\big(\overline{K}_0+\overline{K}_1\Lambda^{\alpha}+\overline{K}_1(\Gamma^{\alpha}+1)\|\nabla f(w)\|^{\alpha}+\overline{K}_2\|w'-w\|^{\frac{\alpha}{1-\alpha}}\big),\nonumber
\end{align}
where (i) uses Lemma \ref{lemma:moment}. The above inequality implies that $f\in\mathcal{L}_{\text{sym}}^*(\alpha)$ based on item \ref{item:df_bound_poly} of Proposition \ref{prop:PS_equiv}.

(Case II) When $\alpha=1$, eq. \eqref{PSavg_exp_up} holds, so we have
\begin{align}
	&\|\nabla f(w')-\nabla f(w)\|=\big\|\mathbb{E}_{\xi}\big(\nabla f_{\xi}(w')-\nabla f_{\xi}(w)\big)\big\| \nonumber\\
	&\le\sqrt{\mathbb{E}_{\xi}\|\nabla f_{\xi}(w')-\nabla f_{\xi}(w)\|^2} \nonumber\\
	&\le\|w'-w\|\sqrt{2L_0^2+4L_1^2\mathbb{E}_{\xi}\|\nabla f_{\xi}(w)\|^2}\exp(6L_1^2\|w'-w\|^2)\nonumber\\
	&\overset{(i)}{\le}\|w'-w\|\sqrt{2L_0^2+4L_1^2\Lambda^2+4L_1^2(\Gamma^{2}+1)\|\nabla f(w)\|^2}\exp(6L_1^2\|w'-w\|^2)\nonumber\\
	&\overset{(ii)}{\le}\|w'-w\|\big(2L_0+2L_1\Lambda+2L_1(\Gamma+1)\|\nabla f(w)\|\big)\exp(6L_1^2\|w'-w\|^2)\nonumber
\end{align}
where (i) uses Lemma \ref{lemma:moment}, (ii) uses the inequality that $\sqrt{a+b}\le \sqrt{a}+\sqrt{b}$ for any $a,b\ge 0$. The above inequality implies that $f\in\mathcal{L}_{\text{sym}}^*(\alpha)$ based on item \ref{item:df_bound_exp} of Proposition \ref{prop:PS_equiv}.

\section{Proof of Theorem \ref{thm:SPIDER}}\label{sec:SPIDERProof}
We will first prove the following lemmas which will be used in the proof of Theorem \ref{thm:SPIDER}.
\begin{lemma}\label{lemma:df_diff_SPIDER}
	Apply SPIDER algorithm (Algorithm \ref{algo: SPIDER}) to $f\in\mathcal{L}_{\text{sym}}^*(\alpha)$ with stepsize $\gamma\le\frac{\epsilon}{2\overline{K}_0+2\overline{K}_2+2\overline{K}_1(\Lambda^{\alpha}+\Gamma^{\alpha}+1)+1}$ (when $\alpha\in(0,1)$) or $\gamma\le\frac{\epsilon}{3L_1\sqrt{\Gamma^2+1}+3\sqrt{L_0^2+2L_1^2\Lambda^2}}$ (when $\alpha=1$) ($\epsilon\in(0,1)$ is the target accuracy). Then we have,
	\begin{align}
		\mathbb{E}_{\xi\sim\mathbb{P}}\big(\|\nabla f_{\xi}({w_{t+1}}) - \nabla f_{\xi}(w_t)\|^2\big|S_{1:t}\big)\le\epsilon^2\big(1+\|\nabla f(w_t)\|^2\big),\label{df_diff_SPIDER}
	\end{align}
	where $\xi\sim\mathbb{P}$ is independent from the minibatches $S_{1:t}$.
\end{lemma}
\begin{proof}[Proof of Lemma \ref{lemma:df_diff_SPIDER}]
	Given $S_{1:t}$, $w_t, w_{t+1}$ are non-random based on eq. \eqref{filter}. Hence, eq. \eqref{PSavg_poly_up} or \eqref{PSavg_exp_up} holds respectively when $\alpha\in(0,1)$ or $\alpha=1$.  
	
	If $\alpha\in (0,1)$, eq. \eqref{df_diff_SPIDER} can be proved as follows
	\begin{align}
		&\mathbb{E}_{\xi\sim\mathbb{P}}\big(\|\nabla f_{\xi}({w_{t+1}}) - \nabla f_{\xi}(w_t)\|^2\big|S_{1:t}\big)\nonumber\\
		&\overset{(i)}{\le} \|w_{t+1}-w_t\|^2\big(\overline{K}_0+\overline{K}_1\mathbb{E}_{\xi}\big(\|\nabla f_{\xi}(w_t)\|^{\alpha}\big|S_{1:t}\big)+\overline{K}_2\|w_{t+1}-w_t\|^{\frac{\alpha}{1-\alpha}}\big)^2\nonumber\\
		&\overset{(ii)}{\le}\gamma^2\big(\overline{K}_0+\overline{K}_1\Lambda^{\alpha}+\overline{K}_1(\Gamma^{\alpha}+1)\|\nabla f(w_t)\|^{\alpha}+\overline{K}_2\big)^2\nonumber\\
		&\overset{(iii)}{\le}2\gamma^2(\overline{K}_0+\overline{K}_2+\overline{K}_1\Lambda^{\alpha})^2+2\gamma^2\overline{K}_1^2(\Gamma^{\alpha}+1)^2\cdot\|\nabla f(w_t)\|^{2\alpha}\nonumber\\
		&\le 2\gamma^2(\overline{K}_0+\overline{K}_2+\overline{K}_1\Lambda^{\alpha})^2+2\gamma^2\overline{K}_1^2(\Gamma^{\alpha}+1)^2\big(\|\nabla f(w_t)\|^2+1\big)\nonumber\\
		&\overset{(iv)}{\le}\epsilon^2\big(1+\|\nabla f(w_t)\|^2\big),\nonumber
	\end{align}
	where (i) uses eq. \eqref{PSavg_poly_up}, (ii) uses eq. \eqref{ESG2} and $\|w_{t+1}-w_t\|=\gamma\le 1$ based on Algorithm \ref{algo: SPIDER}, (iii) uses the inequality that $(a+b)^2\le 2a^2+2b^2$ for any $a,b\ge 0$, (iv) uses $\gamma\le \frac{\epsilon}{2\overline{K}_0+2\overline{K}_2+2\overline{K}_1(\Lambda^{\alpha}+\Gamma^{\alpha}+1)}\le \frac{\epsilon}{2}\big((\overline{K}_0+\overline{K}_2+\overline{K}_1\Lambda^{\alpha})^2+\overline{K}_1^2(\Gamma^{\alpha}+1)^2\big)^{-\frac{1}{2}}$.
	
	If $\alpha=1$, eq. \eqref{df_diff_SPIDER} can be proved as follows 
	\begin{align}
		&\mathbb{E}_{\xi\sim\mathbb{P}}\big(\|\nabla f_{\xi}({w_{t+1}}) - \nabla f_{\xi}(w_t)\|^2\big|S_{1:t}\big)\nonumber\\
		&\overset{(i)}{\le} 2\|w_{t+1}-w_t\|^2\cdot(L_0^2+2L_1^2\mathbb{E}_{\xi}\|\nabla f_{\xi}(w_t)\|^2)\exp(12L_1^2\|w_{t+1}-w_t\|^2)\nonumber\\
		&\overset{(ii)}{=}2\gamma^2\exp(12L_1^2\gamma^2) \big(L_0^2+2L_1^2\Lambda^2+2L_1^2(\Gamma^2+1)\|\nabla f(w_t)\|^2\big)\nonumber\\
		&\overset{(iii)}{\le}\epsilon^2\big(1+\|\nabla f(w_t)\|^2\big),\nonumber
	\end{align}
	where (i) uses eq. \eqref{PSavg_exp_up}, (ii) uses eq. \eqref{ESG2} and $\|w_{t+1}-w_t\|=\gamma\le 1$ based on Algorithm \ref{algo: SPIDER}, (iii) uses $\gamma\le\frac{\epsilon}{3L_1\sqrt{\Gamma^2+1}+3\sqrt{L_0^2+2L_1^2\Lambda^2}}\le\min\big(\frac{1}{3L_1}, \frac{\epsilon}{3L_1\sqrt{\Gamma^2+1}+3\sqrt{L_0^2+2L_1^2\Lambda^2}}\big)$.
\end{proof}

\begin{lemma}\label{lemma:vt_err} Apply SPIDER algorithm (Algorithm \ref{algo: SPIDER}) to $f\in\mathcal{L}_{\text{sym}}^*(\alpha)$ with stepsize $\gamma$ given by Lemma \ref{lemma:df_diff_SPIDER} batchsize $|S_t|=B$ when $t\mod q=0$ and $|S_t|=B'$ otherwise. Then the approximation error $\delta_t:=v_t-\nabla f(w_t)$ has the following properties conditional on minibatches $S_{1:t}:=\{S_1,\ldots,S_t\}$. 
	\begin{align}
		&\mathbb{E}\big(\delta_{t+1}\big|S_{1:t}\big)=\delta_t; \forall (t+1)~\text{mod}~q\ne 0 \label{delta_E1}\\
		&\mathbb{E}\big(\|\delta_{t+1}\|^2\big|w_{t+1}\big)\le\frac{1}{B}\big(\Gamma^2\|\nabla f(w_{t+1})\|^2+\Lambda^2\big); \forall (t+1)~\text{mod}~q=0 \label{delta_qk2}\\
		&\mathbb{E}\big(\|\delta_{t+1}\|^2\big|S_{1:t}\big)\le \|\delta_t\|^2+\frac{\epsilon^2}{B'}\big(1+\|\nabla f(w_t)\|^2\big)\label{delta_E2}
	\end{align}
	Therefore, for any $k\in\mathbb{N}$ and $s=0,1,\ldots,q-1$, we have
	\begin{align}
		\mathbb{E}\|\delta_{qk+s}\|&\le\frac{\Lambda}{\sqrt{B}}+\epsilon\sqrt{\frac{q}{B'}}+\Big(\frac{\epsilon}{\sqrt{B'}}+\frac{\Gamma}{\sqrt{B}}\Big)\sum_{u=0}^{q-1}\mathbb{E}\|\nabla f(w_{qk+u})\|.\label{vt_err}
	\end{align}
\end{lemma}
\begin{proof}[Proof of Lemma \ref{lemma:vt_err}]
	We will first prove eq. \eqref{delta_qk2} when $(t+1)~\text{mod}~q=0$ and then prove eqs. \eqref{delta_E1} \& \eqref{delta_E2} when $(t+1)~\text{mod}~q\ne0$. 
	
	If $(t+1)~\text{mod}~q=0$, then $v_{t+1}=\nabla f_{S_{t+1}}(w_{t+1})$ based on Algorithm \ref{algo: SPIDER}. Hence, eq. \eqref{delta_qk2} can be proved as follows.
	\begin{align}
		\mathbb{E}\big(\|\delta_{t+1}\|^2\big|S_{1:t}\big)=&\mathbb{E}\big(\|\nabla f_{S_{t+1}}(w_{t+1})-\nabla f(w_{t+1})\|^2\big|S_{1:t}\big)\nonumber\\
		=&\frac{1}{|S_t|}\mathbb{E}_{\xi\sim\mathbb{P}}\big(\|\nabla f_{\xi}(w_{t+1})-\nabla f(w_{t+1})\|^2\big|S_{1:t}\big)\nonumber\\
		\overset{(i)}{\le}&\frac{1}{B}\big(\Gamma^2\|\nabla f(w_{t+1})\|^2+\Lambda^2\big), \nonumber
	\end{align}
	where (i) uses Assumption \ref{assum:var}.
	
	If ($t+1)~\text{mod}~q\ne0$, then $v_{t+1}=v_t+\nabla f_{S_{t+1}}(w_{t+1})-\nabla f_{S_{t+1}}(w_t)$ based on Algorithm \ref{algo: SPIDER}. Hence, eq. \eqref{delta_E1} can be proved as follows. 
	\begin{align}
		\mathbb{E}\big(\delta_{t+1}\big|S_{1:t}\big)&=\mathbb{E}\big(v_{t+1}-\nabla f(w_{t+1})\big|S_{1:t}\big)\nonumber\\
		&=\mathbb{E}\big(v_t + \nabla f_{S_{t+1}}({w_{t+1}}) - \nabla {f_{{S_{t+1}}}}(w_t)-\nabla f(w_{t+1})\big|S_{1:t}\big)\nonumber\\
		&\overset{(i)}{=}v_t - \nabla f(w_t)=\delta_t, \nonumber
	\end{align}
	where (i) uses eq. \eqref{filter}. Then eq. \eqref{delta_E2} can be proved as follows. 
	\begin{align}
		&\mathbb{E}\big(\|\delta_{t+1}\|^2\big|S_{1:t}\big)\nonumber\\
		&\overset{(i)}{=}\|\delta_t\|^2+\mathbb{E}\big(\|\delta_{t+1}-\delta_t\|^2\big|S_{1:t}\big)\nonumber\\
		&=\|\delta_t\|^2+\mathbb{E}\big(\|v_{t+1}-v_t-\nabla f(w_{t+1})+\nabla f(w_t)\|^2\big|S_{1:t}\big)\nonumber\\
		&=\|\delta_t\|^2+\mathbb{E}\big(\|\nabla f_{S_{t+1}}({w_{t+1}}) - \nabla f_{S_{t+1}}(w_t)-\nabla f(w_{t+1})+\nabla f(w_t)\|^2\big|S_{1:t}\big)\nonumber\\
		&\overset{(ii)}{=}\|\delta_t\|^2+\frac{1}{|S_{t+1}|}\mathbb{E}_{\xi\sim\mathbb{P}}\big(\|\nabla f_{\xi}({w_{t+1}}) - \nabla f_{\xi}(w_t)-\nabla f(w_{t+1})+\nabla f(w_t)\|^2\big|S_{1:t}\big)\nonumber\\
		&\overset{(iii)}{\le} \|\delta_t\|^2+\frac{1}{|S_{t+1}|}\mathbb{E}_{\xi\sim\mathbb{P}}\big(\|\nabla f_{\xi}({w_{t+1}}) - \nabla f_{\xi}(w_t)\|^2\big|S_{1:t}\big)\nonumber\\
		&\overset{(iv)}{\le} \|\delta_t\|^2+\frac{\epsilon^2}{B'}\big(1+\|\nabla f(w_t)\|^2\big),\nonumber
	\end{align}
	where (i) uses eq. \eqref{delta_E1}, (ii) uses eq. \eqref{filter} which implies that conditional on $S_{1:t}$, $S_{t+1}$ obtained from i.i.d. sampling is the only source of randomness in $\nabla f_{\xi}({w_{t+1}}) - \nabla f_{\xi}(w_t)-\nabla f(w_{t+1})+\nabla f(w_t)$, both (ii) and (iii) use $\mathbb{E}\big(\|\nabla f_{\xi}({w_{t+1}}) - \nabla f_{\xi}(w_t)-\nabla f(w_{t+1})+\nabla f(w_t)\|^2\big|S_{1:t}\big)=0$, and (iv) uses Lemma \ref{lemma:df_diff_SPIDER} and $|S_{t+1}|=B'$ (since $t+1\!\mod q\ne 0$).
	
	Next, to prove eq. \eqref{vt_err}, we will first prove the following relation for any $s,s',k\in\mathbb{N}$ such that $s'\le s\le q-1$.
	\begin{align}
		\mathbb{E}\|\delta_{qk+s}\|\le \mathbb{E}\sqrt{\|\delta_{qk+s'}\|^2+\frac{\epsilon^2(s-s')}{B'}}+\frac{\epsilon}{\sqrt{B'}}\sum_{u=s'}^{s-1}\mathbb{E}\|\nabla f(w_{qk+u})\|.\label{err_induct}
	\end{align}
	We prove eq. \eqref{err_induct} via backward induction on $s'=s,s-1,\ldots,1,0$. Note that eq. \eqref{err_induct} holds trivially for $s'=s$. Then, assume that eq. \eqref{err_induct} holds for a certain value of $s'\in[1,s]$ and we prove eq. \eqref{err_induct} for $s'-1$ as follows.
	\begin{align}
		&\mathbb{E}\|\delta_{qk+s}\|-\frac{\epsilon}{\sqrt{B'}}\sum_{u=s'}^{s-1}\mathbb{E}\|\nabla f(w_{qk+u})\|\nonumber\\
		&\overset{(i)}{\le} \mathbb{E}\mathbb{E}\Bigg(\sqrt{\|\delta_{qk+s'}\|^2+\frac{\epsilon^2(s-s')}{B'}}\Bigg|S_{1:qk+s'-1}\Bigg)\nonumber\\
		&\overset{(ii)}{\le} \mathbb{E}\sqrt{\mathbb{E}\Bigg(\|\delta_{qk+s'}\|^2+\frac{\epsilon^2(s-s')}{B'}\Bigg|S_{1:qk+s'-1}\Bigg)}\nonumber\\
		&\overset{(iii)}{\le} \mathbb{E}\sqrt{\|\delta_{qk+s'-1}\|^2+\frac{\epsilon^2}{B'}\big(1+\|\nabla f(w_{qk+s'-1})\|^2\big)+\frac{\epsilon^2(s-s')}{B'}}\nonumber\\
		&\overset{(iv)}{\le} \mathbb{E}\sqrt{\|\delta_{qk+s'-1}\|^2+\frac{\epsilon^2(s-s'+1)}{B'}}+\frac{\epsilon}{\sqrt{B'}}\mathbb{E}\|\nabla f(w_{qk+s'-1})\|,\nonumber
	\end{align}
	where (i) uses eq. \eqref{err_induct} for $s'$, (ii) applies Jensen's inequality to the concave function $\sqrt{\cdot}$, (iii) uses eq. \eqref{delta_E2}, and (iv) uses the inequality that $\sqrt{a+b}\le \sqrt{a}+\sqrt{b}$ for any $a,b\ge 0$. Substituting $s'=0$ into eq. \eqref{err_induct}, we prove eq. \eqref{vt_err} as follows. 
	\begin{align}
		\mathbb{E}\|\delta_{qk+s}\|&\le \mathbb{E}\sqrt{\|\delta_{qk}\|^2+\frac{s\epsilon^2}{B'}}+\frac{\epsilon}{\sqrt{B'}}\sum_{u=0}^{s-1}\mathbb{E}\|\nabla f(w_{qk+u})\|\nonumber\\
		&\overset{(i)}{\le}\sqrt{\mathbb{E}\|\delta_{qk}\|^2}+\epsilon\sqrt{\frac{q}{B'}}+\frac{\epsilon}{\sqrt{B'}}\sum_{u=0}^{s-1}\mathbb{E}\|\nabla f(w_{qk+u})\|\nonumber\\
		&\overset{(ii)}{\le} \frac{1}{\sqrt{B}}\big(\Gamma\|\nabla f(w_{qk})\|+\Lambda\big)+\epsilon\sqrt{\frac{q}{B'}}+\frac{\epsilon}{\sqrt{B'}}\sum_{u=0}^{s-1}\mathbb{E}\|\nabla f(w_{qk+u})\|\nonumber\\
		&\le\frac{\Lambda}{\sqrt{B}}+\epsilon\sqrt{\frac{q}{B'}}+\Big(\frac{\epsilon}{\sqrt{B'}}+\frac{\Gamma}{\sqrt{B}}\Big)\sum_{u=0}^{q-1}\mathbb{E}\|\nabla f(w_{qk+u})\|,\nonumber
	\end{align}
	where (i) uses the inequality that $\sqrt{a+b}\le \sqrt{a}+\sqrt{b}$ for any $a,b\ge 0$ and then applies Lyapunov inequality, and (ii) uses eq. \eqref{delta_qk2} and then uses $\sqrt{a+b}\le \sqrt{a}+\sqrt{b}$ for any $a,b\ge 0$. 
\end{proof}

\begin{lemma}\label{lemma:df_bound}
	Apply SPIDER algorithm (Algorithm \ref{algo: SPIDER}) to $f\in\mathcal{L}_{\text{sym}}^*(\alpha)$ with batchsize $|S_t|=B$ when $t\mod q=0$ and $|S_t|=B'$ otherwise, and stepsize stepsize $\gamma\le\frac{\epsilon}{2(\overline{K}_0+\overline{K}_1+2\overline{K}_2)+1}$ (when $\alpha\in(0,1)$) or $\gamma\le\frac{\epsilon}{5L_1+8L_0}$ (when $\alpha=1$) ($\epsilon\in(0,1)$ is the target accuracy). Then the  decrease of the function $f$ has the following bound.
	\begin{align}
		f(w_{t+1})-f(w_t)\le \frac{\gamma\epsilon}{8}-\frac{\gamma}{2}\|v_t\|-\frac{3\gamma}{8}\|\nabla f(w_t)\|+\frac{3\gamma}{2}\|v_t-\nabla f(w_t)\|.\label{vt_rate}
	\end{align}
\end{lemma}

\begin{proof}[Proof of Lemma \ref{lemma:df_bound}]
	We consider two cases: $\alpha\in(0,1)$ and $\alpha=1$.
	
	(Case I) When $\alpha\in(0,1)$, eq. \eqref{PS_poly_fup} holds for $f\in\mathcal{L}_{\text{sym}}^*(\alpha)$. Hence, 
	\begin{align}
		&f(w_{t+1})-f(w_t)\nonumber\\
		&\le \nabla f(w_t)^{\top}(w_{t+1}-w_t)+\frac{1}{2}\|w_{t+1}-w_t\|^2\big(\overline{K}_0+\overline{K}_1\|\nabla f(w_t)\|^{\alpha}+2\overline{K}_2\|w_{t+1}-w_t\|^{\frac{\alpha}{1-\alpha}}\big)\nonumber\\
		&=-\frac{\gamma}{\|v_t\|}\nabla f(w_t)^{\top}v_t+\frac{\gamma^2}{2}\big(\overline{K}_0+\overline{K}_1\|\nabla f(w_t)\|^{\alpha}+2\overline{K}_2\gamma^{\frac{\alpha}{1-\alpha}}\big)\nonumber\\
		&\overset{(i)}{\le}-\frac{\gamma\big(v_t-\nabla f(w_t)\big)^{\top}v_t+\gamma\|v_t\|^2}{\|v_t\|}+\frac{\gamma^2}{2}(\overline{K}_0+\overline{K}_1+2\overline{K}_2)+\frac{\overline{K}_1\gamma^2}{2}\|\nabla f(w_t)\|\nonumber\\
		&\overset{(ii)}{\le} \gamma\|v_t-\nabla f(w_t)\|-\frac{\gamma}{2}\|v_t\|-\frac{\gamma}{2}\big(\|\nabla f(w_t)\|-\|v_t-\nabla f(w_t)\|\big)+\frac{\gamma\epsilon}{8}+\frac{\gamma\epsilon}{8}\|\nabla f(w_t)\|\nonumber\\
		&\overset{(iii)}{\le} \frac{\gamma\epsilon}{8}-\frac{\gamma}{2}\|v_t\|-\frac{3\gamma}{8}\|\nabla f(w_t)\|+\frac{3\gamma}{2}\|v_t-\nabla f(w_t)\|, \label{vt_rate_poly}
	\end{align}
	where (i) uses $\|\nabla f(w_t)\|^{\alpha}\le \|\nabla f(w_t)\|^2+1$ and $\gamma\le 1$, (ii) uses Cauchy-Schwartz inequality, $\|v_t\|\ge\|\nabla f(w_t)\|-\|v_t-\nabla f(w_t)\|$ and $\gamma\le\frac{\epsilon}{2(\overline{K}_0+\overline{K}_1+2\overline{K}_2)}$, (iii) uses $\epsilon\le 1$. 
	
	(Case II) When $\alpha=1$, we have $\beta=1$ and eq. \eqref{PS_exp_fup} holds for $f\in\mathcal{L}_{\text{sym}}^*(1)$. Hence, 
	\begin{align}
		&f(w_{t+1})-f(w_t)\nonumber\\
		&\le \nabla f(w_t)^{\top}(w_{t+1}-w_t)+\frac{1}{2}\|w_{t+1}-w_t\|^2\big(L_0+L_1\|\nabla f(w_t)\|\big)\exp\big(L_1\|w_{t+1}-w_t\|\big)\nonumber\\
		&\le -\frac{\gamma}{\|v_t\|}\nabla f(w_t)^{\top}v_t+\frac{1}{2}\gamma^2\big(L_0+L_1\|\nabla f(w_t)\|\big)\exp(L_1\gamma)\nonumber\\
		&\overset{(i)}{\le}-\frac{\gamma\big(v_t-\nabla f(w_t)\big)^{\top}v_t+\gamma\|v_t\|^2}{\|v_t\|}+L_0\gamma^2+\frac{\gamma}{8}\|\nabla f(w_t)\|\nonumber\\
		&\overset{(ii)}{\le} \gamma\|v_t-\nabla f(w_t)\|-\frac{\gamma}{2}\|v_t\|-\frac{\gamma}{2}\big(\|\nabla f(w_t)\|-\|v_t-\nabla f(w_t)\|\big)+\frac{\gamma\epsilon}{8}+\frac{\gamma}{8}\|\nabla f(w_t)\|\nonumber\\
		&\le \frac{\gamma\epsilon}{8}-\frac{\gamma}{2}\|v_t\|-\frac{3\gamma}{8}\|\nabla f(w_t)\|+\frac{3\gamma}{2}\|v_t-\nabla f(w_t)\|,\label{vt_rate_exp}
	\end{align}
	where (i) uses $\gamma\le \frac{1}{5L_1}$, (ii) uses Cauchy-Schwartz inequality, $\|v_t\|\ge\|\nabla f(w_t)\|-\|v_t-\nabla f(w_t)\|$ and $\gamma\le\frac{\epsilon}{8L_0}$.
\end{proof}

Now we will prove Theorem \ref{thm:SPIDER}. First, it can be easily verified that the choice of stepsize $\gamma$ and batchsize $|S_t|$ satisfies the requirements of Lemmas \ref{lemma:vt_err} \& \ref{lemma:df_bound}. Therefore, eq. \eqref{vt_rate} in Lemma \ref{lemma:df_bound} holds. Taking expectation of eq. \eqref{vt_rate} and telescoping over $t=0,1,\ldots,T-1$ where $T=qK$, we obtain that
\begin{align}
	&\mathbb{E}f(w_T)-\mathbb{E}f(w_0)\nonumber\\
	&\le \frac{T\gamma\epsilon}{8}-\frac{\gamma}{2}\sum_{t=0}^{qK-1}\mathbb{E}\|v_t\|-\frac{3\gamma}{8}\sum_{t=0}^{qK-1}\mathbb{E}\|\nabla f(w_t)\|+\frac{3\gamma}{2}\sum_{t=0}^{qK-1}\mathbb{E}\|v_t-\nabla f(w_t)\|\nonumber\\
	&\le \frac{T\gamma\epsilon}{8}-\frac{\gamma}{2}\sum_{k=0}^{K-1}\sum_{s=0}^{q-1}\mathbb{E}\|v_{qk+s}\|-\frac{3\gamma}{8}\sum_{k=0}^{K-1}\sum_{s=0}^{q-1}\mathbb{E}\|\nabla f(w_{qk+s})\|+\frac{3\gamma}{2}\sum_{k=0}^{K-1}\sum_{s=0}^{q-1}\mathbb{E}\|\delta_{qk+s}\|\nonumber\\
	&\overset{(i)}{\le} \frac{T\gamma\epsilon}{8}-\frac{\gamma}{2}\sum_{k=0}^{K-1}\sum_{s=0}^{q-1}\mathbb{E}\|v_{qk+s}\|-\frac{3\gamma}{8}\sum_{k=0}^{K-1}\sum_{s=0}^{q-1}\mathbb{E}\|\nabla f(w_{qk+s})\|\nonumber\\
	&\quad+\frac{3q\gamma}{2}\sum_{k=0}^{K-1}\Bigg(\frac{\Lambda}{\sqrt{B}}+\epsilon\sqrt{\frac{q}{B'}}+\Big(\frac{\epsilon}{\sqrt{B'}}+\frac{\Gamma}{\sqrt{B}}\Big)\sum_{u=0}^{q-1}\mathbb{E}\|\nabla f(w_{qk+u})\|\Bigg)\nonumber\\
	&\overset{(ii)}{\le} \frac{T\gamma\epsilon}{4}-\frac{5\gamma}{16}\sum_{k=0}^{K-1}\sum_{s=0}^{q-1}\mathbb{E}\|\nabla f(w_{qk+s})\|,\label{SPIDER_fulldec}
\end{align}
where (i) uses eq. \eqref{vt_err} and (ii) uses the following condition satisfied by the hyperparamter choices $B\ge \max(576\Lambda^2\epsilon^{-2},2304\Gamma^2q^2), B'\ge\max(576q,2304q^2\epsilon^2)$.
\begin{align}
	&\frac{\Lambda}{\sqrt{B}}+\epsilon\sqrt{\frac{q}{B'}}\le \frac{\epsilon}{12}, \quad \frac{\epsilon}{\sqrt{B'}}+\frac{\Gamma}{\sqrt{B}}\le\frac{1}{24q}\nonumber
\end{align}
By rearranging eq. \eqref{SPIDER_fulldec} and using $f(w_T)\ge f^*=\min_{w\in\mathbb{R}^d}f(w)$, we can prove eq. \eqref{df_rate} as follows
\begin{align}
	&\mathbb{E}\|\nabla f(w_{\widetilde{T}})\|=\frac{1}{T}\sum_{t=0}^{T-1}\|\nabla f(w_t)\|\le \frac{16}{5T\gamma}\big(\mathbb{E}f(w_0)-f^*\big)+\frac{4\epsilon}{5}\nonumber
\end{align}

It can be easily verified that the following hyperparameter choices satisfy the condition that $B\ge \max(576\Lambda^2\epsilon^{-2},2304\Gamma^2q^2)$ and that $B'\ge\max(576q,2304q^2\epsilon^2)$ since $\epsilon\in(0,1)$.
\begin{align}
	q&=\epsilon^{-1}=\mathcal{O}(\epsilon^{-1})\label{q}\\
	B&=\max(576\Lambda^2,2304\Gamma^2)\epsilon^{-2}=\mathcal{O}(\epsilon^{-2})\label{S0}\\
	B'&=2304\epsilon^{-1}=\mathcal{O}(\epsilon^{-1})\label{S1}\\
	\gamma&=\frac{\epsilon}{2\overline{K}_0+4\overline{K}_2+2\overline{K}_1(\Lambda^{\alpha}+\Gamma^{\alpha}+1)+1}=\mathcal{O}(\epsilon); \quad \text{if }  \alpha\in(0,1)\label{gamma1}\\
	\gamma&=\frac{\epsilon}{5L_1\sqrt{\Gamma^2+1}+8\sqrt{L_0^2+2L_1^2\Lambda^2}}=\mathcal{O}(\epsilon); \quad \text{if } \alpha=1\label{gamma2}\\
	K&=\frac{16\epsilon}{5T\gamma}\big(\mathbb{E}f(w_0)-f^*\big)=\mathcal{O}(\epsilon^{-1})\label{K}\\
	T&=qK=\frac{16}{5T\gamma}\big(\mathbb{E}f(w_0)-f^*\big)=\mathcal{O}(\epsilon^{-2})\label{T}
\end{align}
Substituting the choice of $T$ given by eq. \eqref{T} into eq. \eqref{vt_rate}, we obtain that $\mathbb{E}\|\nabla f(w_{\widetilde{T}})\|\le\epsilon$. 

Under the above hyperparameter choices, the sample complexity is 
\begin{align}
	\sum_{t=0}^{qK-1}|S_t|=K\big((q-1)B'+B\big)=\mathcal{O}(\epsilon^{-3}).\nonumber
\end{align}

\end{document}